\documentclass{amsart}
\usepackage[T1]{fontenc}
\usepackage[utf8]{inputenc}
\usepackage{fullpage}
\usepackage[top=2cm, bottom=4.5cm, left=2.5cm, right=2.5cm]{geometry}
\usepackage{amsmath,amsthm,amsfonts,amssymb,amscd}
\usepackage{lastpage}
\usepackage{enumerate}
\usepackage{longtable}
\usepackage{fancyhdr}
\usepackage{mathrsfs}
\usepackage{xcolor}
\usepackage{graphicx}
\usepackage{listings}
\usepackage{hyperref}
\usepackage{tikz}
\usepackage{tikz-cd}
\usepackage{commath}
\usepackage{bbm}
\usepackage{multirow}
\usepackage{url}
\usepackage{multicol}
\usepackage{enumitem}
\usepackage{pgfplots}
\usepackage{cleveref}
\usepgfplotslibrary{polar}
\usepgflibrary{shapes.geometric}
\usetikzlibrary{calc}
\usepgfplotslibrary{fillbetween}
\usetikzlibrary{patterns}
\usetikzlibrary{calc}
\usetikzlibrary{decorations.pathmorphing}
\usepackage{quiver}
\usepackage{stmaryrd}
\usepackage{placeins}
\usepackage[backend=biber]{biblatex}
\usepackage{comment}

\newcommand{\p}{\partial}

\newcommand{\Z}{\mathbb{Z}}

\newcommand{\R}{\mathbb{R}}
\newcommand{\C}{\mathbb{C}}

\newcommand{\ra}{\rightarrow}

\newcommand{\id}{\text{id}}

\newtheorem{theorem}{Theorem}
\newtheorem{corollary}[theorem]{Corollary}
\newtheorem{lemma}[theorem]{Lemma}
\newtheorem{definition}[theorem]{Definition}
\newtheorem{prop}[theorem]{Proposition}
\newtheorem*{Acknow}{Acknowledgements}
\newtheorem{remark}[theorem]{Remark}
\newcommand{\A}{\mathcal{A}}
\DeclareMathOperator{\dee}{\mathfrak{d}}
\DeclareMathOperator{\deebar}{\overline{\mathfrak{d}}}
\DeclareMathOperator{\Prim}{\mathcal{P}}

\theoremstyle{definition}

\DeclareMathOperator{\End}{End}

\DeclareMathOperator{\im}{im}

\newcommand{\lmod}{{\operatorname{\mathbf{-mod}}}}

\newtheorem*{theorem*}{Theorem}

\title{Symplectic Hodge Theory on Lie Algebroids}
\author{Rankin, Shane}
\pgfplotsset{compat=1.18}
\begin{document}

\begin{abstract}
    We explore the natural analogues of the Brylinski condition, strong Lefschetz condition, and $d\delta$-lemma in symplectic geometry originally explored by Brylinski, Mathieu, Yan, and Guillemin to the symplectic Lie algebroid case. The equivalence of the three conditions is re-established as a purely algebraic statement along with a primitive notion of the $d\delta$-lemma shown establsihed by Tseng, Yau, and Ho. We then show that the natural analogues of these in the Lie algebroid setting holds as well with examples given.
\end{abstract}
\maketitle

\tableofcontents
\section{Introduction}
\subsection{Background}
Hodge theory for symplectic manifolds was introduced by Ehresmann and Liberman and rediscovered by Brylinski \cite{BrylinskiPoisson}. Brylinski conjectured that on a symplectic manifold $(M,\omega)$ of dimension $2m$, every deRham cohomology class admits a symplectic harmonic representative. Mathieu \cite{Mathieu} and Yan \cite{YAN1996143} independently showed that the Brylinski conjecture holds true if and only if $(M^{2m},\omega)$ satisfies the strong Lefschetz property, that is for all $0 \leq k \leq m$, the map 
\begin{align*}
    [L]^k:H^{m-k}(M) &\ra H^{m+k}(M) \\
    [\alpha] &\mapsto [\omega^k \wedge \alpha]
\end{align*}
is surjective. This result was then improved by Merkulov in \cite{merkulov1998formality}, and Guillemin in \cite{Guilleminddlemma} by showing that these two conditions are equivalent to the symplectic $d \delta$-lemma in the case that $M$ was a compact manifold using Poincar\'e duality. In a key departure from traditional Riemannian Hodge theory, the symplectic Laplacian is no longer an elliptic operator as it vanishes identically, and as such the results that heavily makeup the theory no longer work. Instead, we primarily work with the observation that in the presence of a symplectic form, the space of differential forms on a smooth manifold admits the structure of an $\mathfrak{sl}_2(\R)$-module. This ``symplectic Hodge theory'' has been extended to new settings in work by Lin \cite{equivariantddlemma} and others, which require reestablishing many conditions in specific geometric settings. In this paper we work towards streamlining this by proving analogous statements about modules in a relevant module category, and applying them to symplectic Lie algebroids, a minimal smooth setting for the notions involved to make sense. We also generalize the operators $\p_+,\p_-$ of Tseng and Yau \cite{tseng2012cohomologyhodgetheorysymplectic} to this strictly algebraic setting (which we rename $\dee$ and $\deebar$ respectively), though we do not explore the related Laplacians, or Appelli and Bott-Chern cohomologies corresponding to certain combinations of these operators. We then relate the original trio of equivalences back to $\p_+\p_-$-lemma of \cite{HoSym} by reestablishing the equivalence of this to the original triple in our algebraic framework rather than the original topological framework considered. Through this generalization we are also able to lift the compactness condition, though it was free to assume in the classical case of the tangent bundle by Poincar\'e Duality, and standard algebraic topological facts relating to the relevant conditions.
\subsection{Results}
We define a category of modules $\mathscr{C}_n$ over a suitable Lie Superalgebra with properties that axiomatize the algebra of differential forms on a symplectic Manifold much in the spirit of \cite{MR1480723}.
\begin{definition}
    Let $\mathfrak{g}$ be the 5 dimensional Lie Superalgebra with basis $e,f,h,d,\delta$ with $e,f,h$ degrees $2,-2,0$ respectively and $d,\delta$ in degrees $1,-1$ respectively
    all subject to the following relations:
    \begin{align*}
        [e,f]&=h & [e,h]&=-2e & [f,h]&=2f \\
        [e,d]&=0 & [f,d]&=\delta & [h,d]&=d \\
        [e,\delta]&=d & [f,\delta]&=0 & [h,\delta]&=-\delta  \\
        [d,d]&=0 & [d, \delta]&=0 & [\delta,\delta]&=0 
    \end{align*}
    where all commutators are graded.
\end{definition}
It is well known that the space of forms of any symplectic manifold admits the structure of a $\mathfrak{g}$-module where the even part of the algebra corresponds to the standard $\mathfrak{sl}_2$ structure, $d$ corresponds to the deRham differential, and $\delta$ is defined in \cite{BrylinskiPoisson}. The algebra of differential forms on a manifold of dimension $2n$ admits finitely many $\Omega^k(M)$ such that $\Omega^k(M) \neq \{0\}$, namely when $k$ is between $0$ and $2n$, and so we restrict our attention to modules for which $h$ acts with finitely many eigenvalues. Moreover, since $h$ acts integrally and via a diagonalizable action, we want to restrict to modules for which this occurs as well.
\begin{definition}
    For $n \geq 0$, let $\mathscr{C}_{n}$ denote the full subcategory of $\mathfrak{g}\lmod$ such that $h$ acts integrally via diagonalizable action with finitely many eigenvalues between $-n$ and $n$.
\end{definition}
 Objects of this category are also referred to as ``Modules of Finite $h$-type'' or ``Modules of Finite $h$-spectrum'' in the literature. In this algebraic setting, we recover the following theorems:
\begin{theorem*}
     Let $V \in \mathrm{Ob}(\mathscr{C}_n)$. Then the following are equivalent:
    \begin{enumerate}
        \item The map $[e^i]:H^{-i}(V,d) \ra H^i(V,d)$ is an isomorphism for all $i$.
        \item The inclusion $(\ker(\delta),d) \hookrightarrow (V,d)$ is a quasi-isomorphism.
         \item The $d\delta$-lemma holds for $V$.
        \item The $\dee \deebar$-lemma holds for $V$.
    \end{enumerate}
\end{theorem*}
After introducing the notion of the $d \delta$-lemma holding weakly in section \ref{Weak}, we also weaken this to the following theorem:
\begin{theorem*}
    Let $V \in \mathrm{Ob}(\mathscr{C}_n)$. Then the following are equivalent:
    \begin{enumerate}
        \item The $d\delta$-lemma holds for $V$ up to degree $s$.
        \item The map $[e^i]:H^{-i}(V,d) \ra H^i(V,d)$ is an isomorphism for all $|i|\geq s$.
        \item The inclusion $(\ker(\delta),d) \hookrightarrow (V,d)$ is a quasi-isomorphism in degree $|i| \geq s$.
    \end{enumerate}
\end{theorem*}
Applying the first of these two theorems to a symplectic Lie algebroid with representation we have the following theorem:
\begin{theorem*}
    Let $(\A \ra M, \rho, \omega)$ be a symplectic Lie algebroid of rank $2m$, with representation $(E,\nabla)$. Then the following are equivalent:
    \begin{enumerate}
        \item The map $[L]^k:H^{m-k}_\A(M,E) \ra H^{m+k}_\A(M,E)$ given by $[\omega]^k \wedge -$ is an isomorphism for all $k$.
        \item The inclusion of complexes $(\ker({}^\A\nabla^*), {}^\A \nabla) \hookrightarrow (\Omega_\A^\bullet(M,E),{}^\A \nabla)$ is a quasi-isomorphism.
        \item The $(\A,E)$-$d\delta$-lemma holds for $\A$.
        \item The $(\A,E)$-$\dee \deebar$-lemma holds for $\A$.
    \end{enumerate}
\end{theorem*}
Choosing $E$ to be the trivial line bundle on $M$, $\A = TM$, and with the standard deRham differential $d$ as the connection we recover the main results of \cite{YAN1996143}, \cite{Mathieu}, and \cite{HoSym}. The paper is outlined as follows: In section $2$ we introduce the relevant Lie superalgebra and associated module category and go on to establish the four results equivalence as a purely algebraic fact, as well as reestablish the notion of the ``weak'' Lefschetz property as seen in \cite{fernández2005ddeltalemmaweaklylefschetzsymplectic}. In section $3$ we introduce symplectic Lie algebroids, and define operators on the space of $\A$-forms analogous to those on a symplectic Manifold. Section $4$ contains examples of the theorems proven. First, the classical example $KT^4$ - The Kodaira-Thurston manifold, fits well into the framework as it should. Next, using Ovando's \cite{ovando2004dimensional} classification of symplectomorphism classes of $4$-dimensional Lie Algebras, we discuss which admit the structure of a Lefschetz module. Moving on we consider six-dimensional Nilmanifolds, and discuss which of the 26 isomorphism classes admitting symplectic forms admit Lefschetz maps that are isomorphisms. We then discuss $E$-manifolds as studied by Miranda and Scott \cite{Miranda_2020}, and finally K\"ahler Lie algebroids.
\begin{Acknow}
   We would like to thank the referee for a number of insightful comments and suggestions that led to substantial simplifications and a clearer presentation of the main results.
\end{Acknow}
\section{Lefschetz Modules}
\subsection{The Lie Superalgebra}
 If $V \in \mathrm{Ob}(\mathscr{C}_{n})$, then $V$ admits a decomposition
    \begin{equation*}
        V = \bigoplus_{i=-n}^n V_i,
    \end{equation*}
 where $h$ acts on $V_i$ as scaling by $i$, i.e. $i\cdot \id_{V_i}$.
 First, we list a number of identities derived from the bracket relations that will be used later. 
\begin{prop} For any $V \in \mathrm{Ob}(\mathscr{C}_n)$, we have the relations:
\begin{align*}
    [f^k,d]=kf^{k-1}\delta, \quad [e^k,h]&=-2ke^k, \quad [e^k,f]=ke^{k-1}(h+k-1), \\
    [e^k,\delta]=ke^{k-1}d, \quad & \quad  [h,e^k] = 2ke^k.
\end{align*}
\end{prop}
\begin{proof}
    All of these follow by induction, with the base cases coming from the defining relations of $\mathfrak{g}$.
\end{proof}

\begin{prop}\label{ddeltaebracket0}
    $[d\delta,e]=0$.
\end{prop}
\begin{proof}
    We have that 
    \begin{equation*}
        [d\delta,e] = d(e\delta-d)-ed\delta = de\delta - d^2 -d e \delta =0,
    \end{equation*}
    where we use $[d,e]=[d,d]=0$.
\end{proof}

\begin{lemma}\label{Cyclicfindim}
    Every cyclic $V \in \mathrm{Ob}(\mathscr{C}_n)$ is finite-dimensional.
\end{lemma}
\begin{proof}
Let $v$ generate $V$ as a $U\mathfrak{g}$ module, and decompose $v = \sum v_i$ for $v_i \in V_i$. It suffices to show that the module generated by such a $v_i$ is finite-dimensional. By the PBW theorem for Lie Superalgebras, see \cite[Theorem 6.1.1]{MR2906817}, it suffices to show that there are finitely many nonzero expressions of the form
\begin{equation*}
    e^\ell f^k h^r d^s \delta^t v_i.
\end{equation*}
First, we have that $0 \leq s,t <2$ by the bracket relations. Next, note that $h(d^s\delta^tv_i)=(s-t+i)d^s\delta^tv_i$, and since $h$ has finitely many nonzero eigenvalues, there are finitely many nonzero such expressions. Moving on, we have that
\begin{equation*}
    h(f^k d^s \delta^t v_i) = (s-t-2k+i) f^k d^s \delta^t v_i,
\end{equation*}
and since $0 \leq k \leq n$ and $h$ has finitely many non-zero eigenvalues, again there are only finitely many non-zero  expressions of the above form. Finally, we have that
\begin{equation*}
    h(e^\ell f^k d^s \delta^t v_i) = (2\ell -2k+s-t+i)e^\ell f^k d^s \delta^t v_i,
\end{equation*}
of which there are finitely many expressions that are non-zero by the same argument as before along with the fact that $0 \leq k \leq n$.
\end{proof} 
  
 \begin{definition}
     Suppose $V \in \mathrm{Ob}(\mathscr{C}_n)$ with decomposition $V = \bigoplus_i V_i$. We call a vector $v_i \in V_i$ primitive if $fv_i=0$ or equivalently $e^{i+1}v_i=0$.
 \end{definition}
 We now recall a few results from the structure theory of $\mathfrak{sl}_2$-modules:
 \begin{theorem}
     Let $V \in \mathrm{Ob}(\mathscr{C}_n)$. Then every $v \in V$ can be written uniquely as a finite sum
         \begin{equation*}
             v = \sum e^i v_i,
         \end{equation*}
         where the $v_i$ are primitive.
 \end{theorem}
 \begin{remark}
     This only holds since the number of nonzero eigenvalues is bounded in both directions for any object of $\mathscr{C}_m$. Without this we have no expectation for such a result.
 \end{remark}
 \begin{prop}
     Suppose $V \in \mathrm{Ob}(\mathscr{C}_n)$ with decomposition $V = \bigoplus_i V_i$. Then the map $e^i:V_{-i} \ra V_i$ is a bijection.
 \end{prop}
 \begin{prop}
    For primitive $v \in V_{-i}$, we have that
    \begin{equation*}
        f^ke^k v= \lambda_{k,i} v,
    \end{equation*}
    where
    \begin{equation*}
        \lambda_{k,i} = \frac{k! i!}{(i-k)!}.
    \end{equation*}
\end{prop}
\begin{proof}
    We have that
    \begin{equation*}
        fe^kv =  e^kf v - ke^{k-1}(h-k+1)v = k(i+k-1)e^{k-1}v,
    \end{equation*}
    since $fv=0$ as $v$ is primitive. Applying $f$ to this expression again, we see that
    \begin{equation*}
        f^2e^kv = k(k-1)(i-k+1)(i-k+2) e^{k-2}v.
    \end{equation*}
    Continuing, we have that
    \begin{equation*}
        f^ke^k v = k(k-1)...(k-(k-1)) (i-k+1)(i-k+2) \dots (i-k+k) v = \frac{k! i!}{(i-k)!}v
    \end{equation*}
    which is well defined, as if $k> i$, then $e^kv=0$, so $k\leq i$.
\end{proof}
\begin{prop}\label{e^kf^ke^rv_i}
    Let $k,r,i\geq 0$ and $v \in V_{-i}$ be primitive. Then
    \begin{align*}
        f^ke^k(e^rv) = \frac{(r+k)!(i-r)!}{r!(i-r-k)!} e^rv.
    \end{align*}
\end{prop}
\begin{proof}
    This follows by induction on $r$, where the base case is given by the preceding proposition.
\end{proof}
We'd like to introduce the notion of the symplectic Hodge star, though the typical method of construction is unavailable. Instead, note that since we have unique primitive decompositions, it suffices to define a star on elements of the form $e^pv_{-r}$ for $0 \leq p \leq r$ and $v_{-r}\in V_{-r}\cap \ker(f)$.
\begin{definition}
    Fix $e^pv_{-r}$ for $v_{-r} \in V_{-r}\cap \ker(f)$ and $0 \leq p \leq r$. Define the ``star operator'' $\star$ via:
    \begin{equation*}
        \star (e^pv_{-r}) := \begin{cases}
            (-1)^{\frac{r(r-1)}{2}} \frac{p!}{(r-p)!}e^{r-2p} (e^pv_{-r}) &\text{if }2p\leq r \\
            (-1)^{\frac{r(r-1)}{2}}\frac{(r-p)!}{p!} f^{2p-r}(e^pv_{-r}) &\text{if }2p>r.
        \end{cases}
    \end{equation*}
\end{definition}
\begin{prop}
    $\star^2=\id$.
\end{prop}
\begin{proof}
    Suppose $e^pv_{-r}$ is such that $2p \leq r$. Then
    \begin{equation*}
        \star(e^pv_{-r}) = (-1)^{\frac{r(r-1)}{2}} \frac{p!}{(r-p)!}e^{r-p}v_{-r},
    \end{equation*}
    from which two cases arise. First, if $r-p=r/2$, then $r=p$ and we have that
    \begin{equation*}
        \star^2(e^pv_{-r})=(-1)^{\frac{r(r-1)}{2}} \frac{p!}{(r-p)!} (-1)^{\frac{r(r-1)}{2}} \frac{0!}{r!}e^pv_{-r} = e^pv_{-r}.
    \end{equation*}
    The other case is when $r-p>r/2$, where using Proposition \ref{e^kf^ke^rv_i} we have
    \begin{align*}
        \star^2(e^pv_{-r}) &= (-1)^{\frac{r(r-1)}{2}} \frac{p!}{(r-p)!}\star (e^{r-p}v_{-r}) \\
        &= (-1)^{\frac{r(r-1)}{2}} \frac{p!}{(r-p)!}\left((-1)^{\frac{r(r-1)}{2}} \frac{(r-(r-p))!}{(r-p)!}f^{2(r-p)-r}e^{r-p}v_{-r} \right) \\
        &=\left(\frac{p!}{(r-p)!} \right)^2 f^{r-2p}e^{r-2p}e^pv_{-r} \\
        &=\left(\frac{p!}{(r-p)!} \right)^2 \frac{(p+(r-2p))! (r-p)!}{p! (r-p-(r-2p))!}e^pv_{-r} \\
        &= e^pv_{-r},
    \end{align*}
    thus if $2p\leq r$, we have that $\star^2(e^pv_{-r})=e^pv_{-r}$. The case when $2p>r$ follows similarly, though case analysis is not needed.
\end{proof}
With this, we can extend the star linearly to obtain a linear map $\star:V \ra V$ that sends $V_{-i}$ to $V_i$. Moreover, if we define subspaces
\begin{equation*}
    V_{p,r}:= \{e^pv_{-r}| 0 \leq p \leq r \} \subset V_{2p-r}
\end{equation*}
then the star above gives us an isomorphism
\begin{align*}
    \star:V_{p,r} &\ra V_{r-p,r} \\
    e^pv_{-r} &\mapsto (-1)^{\frac{r(r-1)}{2}} \frac{p!}{(r-p)!}e^{r-p}v_{-r}.
\end{align*}
\begin{remark}
    These $V_{p,r}$ are typically denoted by $\mathcal{L}_{r,s}$ or $\mathcal{L}_{p,q}$ in the literature, see \cite{tseng2012cohomologyhodgetheorysymplectic}, \cite{HoSym}.
\end{remark}
We also claim that $\star d\star=\pm \delta$, where sign depends on the degree of the argument. To show this, we need a technical lemma first that will show up later as well.
\begin{lemma}
    Suppose that $w_i \in V_i \cap \ker(f)$. Then the primitive decomposition of $dw_i$ only has at most two nonzero terms.
\end{lemma}
\begin{proof}
    First, note that 
    \begin{equation*}
        f^2dw_i = f(df+\delta)w_i = f\delta w_i = \delta f w_i =0,
    \end{equation*}
    as $w_i$ is primitive. Taking a primitive decomposition of $dw_i$ we have
    \begin{equation*}
        dw_i = \sum_{k \geq 0}e^k v_{i+1-2k},
    \end{equation*}
    so applying $f^2$ we have
    \begin{equation*}
        0 = \sum_{k\geq 0} f^2 e^kv_{i+1-2k}
    \end{equation*}
    where $v_{i+1-2k}\in V_{i+1-2k}\cap \ker(f)$.
    This means for all $k \geq 2$, we had to have had that $e^kv_{i+1-2k}=0$, but as $e^i$ acts injectively, we have that $v_{i+1-2k}=0$ for these $k$, leaving us with
    \begin{equation*}
        dw_i = v_{i+1}+ev_{i-1}.
    \end{equation*}
\end{proof}
\begin{definition}
    Given $v_i \in V_i\cap \ker(f)$ with $dv_i = v_{i+1}+ev_{i-1}$, define
    \begin{align*}
        \dee v_i = v_{i+1}, \quad 
        \deebar  v_i = v_{i-1}.
    \end{align*}
\end{definition}
These are the same operators as $\p_+$ and $\p_-$ as introduced in \cite{tseng2012cohomologyhodgetheorysymplectic} where $\dee$ corresponds to $\p_+$ and $\deebar$ corresponds to $\p_-$. Using these two operators along with the more general form of the star isomorphism, we have the following previous claimed fact.

\begin{prop}
    For any $v\in V_{i}$, we have that $\delta v = (-1)^{i+1}\star d\star v$.
\end{prop}
\begin{proof}
    It suffices to check this on elements of the form $e^pv_{-r}$ where $v_{-r}$ is primitive, and $2p-r=i$. On one hand we have that
    \begin{align*}
        \delta (e^pv_{-r}) &= (fd-df)e^pv_{-r} \\
        &=fe^p dv_{-r} - d(e^pf - pe^{p-1}(h+p-1))v_{-r} \\
        &= fe^p dv_{-r} +p(p-r-1)e^{p-1}dv_{-r} \\
        &=(e^pf - pe^{p-1}(h+k-1))(\dee v_{-r} + e\deebar v_{-r}) + p(p-r-1)e^{p-1}(\dee v_{-r}+e \deebar v_{-r}) \\
        &=-pe^{p-1}\dee v_{-r}+ (r+1-p)e^p \deebar v_{-r}.
    \end{align*}
    On the other hand, the $\star$ isomorphism gives us
    \begin{align*}
        \star d \star (e^pv_{-r}) &= (-1)^{\frac{r(r-1)}{2}}\frac{p!}{(r-p)!}\left( (-1)^{\frac{(r+1)r}{2}}\frac{(r-p)!}{(p-1)!}e^{p-1}\dee v_{-r} + (-1)^{\frac{(r-1)(r-2)}{2}} \frac{(r-p+1)!}{p!} e^{p}\deebar v_{-r} \right) \\
        &=(-1)^{r^2}pe^{p-1}\dee v_{-r} + (-1)^{(r-1)^2}(r-p+1) e^p \deebar v_{-r} \\
        &=(-1)^{r+1}(-pe^{p-1}\dee v_{-r} +(r+1-p)e^p \deebar v_{-r}) \\
        &=(-1)^{i+1} \delta(e^pv_{-r}).
    \end{align*}
    
\end{proof}

  Now with our star established, we make the following definitions in the spirit of classical Hodge theory.
 \begin{definition}
    Let $V \in \mathscr{C}_n$, and $v \in V$. Then we define the following:
    \begin{itemize}
        \item $v$ is ``closed'' if $v \in \ker(d)$.
        \item $v$ is ``exact'' if $v \in \im(d)$.
        \item $v$ is ``coclosed'' if $v \in \ker(\delta)$.
        \item $v$ is ``coexact'' if $v \in \im (\delta)$.
    \end{itemize}
\end{definition}
 We also adopt the language of harmonicity as follows:
 \begin{definition}
     Let $V \in \mathrm{Ob}(\mathscr{C}_n)$, and fix $v \in V$. We say that $v$ is ``harmonic'' if $v \in \ker(d) \cap \ker(\delta)$, and we denote the submodule of harmonic elements by $\widehat{V}$.
 \end{definition}
\begin{prop}
 For $V \in \mathrm{Ob}(\mathscr{C}_n)$, $\widehat{V}$ is a $\mathfrak{g}$-submodule of $V$.
\end{prop}
\begin{proof}
    This is clearly a vector subspace and closed under the action of $d$ and $\delta$. If $v \in \widehat{V}$, then $dev=edv=0$ and $\delta ev = e\delta v -dv=0$. Likewise $dfv = fdv+\delta v=0$ and $\delta f v= f \delta v=0$. The closedness under $h$ follows from $e$ and $f$.
\end{proof}
One can regard $V$ as a bi-differential $\Z$-graded complex after writing $V$ out as
 % https://q.uiver.app/#q=WzAsNSxbMiwwLCJWX2kiXSxbMywwLCJWX3tpKzF9Il0sWzEsMCwiVl97aS0xfSJdLFswLDAsIlxcZG90cyJdLFs0LDAsIlxcZG90cyJdLFszLDIsImQiLDAseyJvZmZzZXQiOi0xfV0sWzIsMCwiZCIsMCx7Im9mZnNldCI6LTF9XSxbMCwxLCJkIiwwLHsib2Zmc2V0IjotMX1dLFsxLDQsImQiLDAseyJvZmZzZXQiOi0xfV0sWzQsMSwiXFxkZWx0YSIsMCx7Im9mZnNldCI6LTF9XSxbMSwwLCJcXGRlbHRhIiwwLHsib2Zmc2V0IjotMX1dLFswLDIsIlxcZGVsdGEiLDAseyJvZmZzZXQiOi0xfV0sWzIsMywiXFxkZWx0YSIsMCx7Im9mZnNldCI6LTF9XV0=
\[\begin{tikzcd}[ampersand replacement=\&]
	\dots \& {V_{i-1}} \& {V_i} \& {V_{i+1}} \& \dots
	\arrow["d", shift left, from=1-1, to=1-2]
	\arrow["d", shift left, from=1-2, to=1-3]
	\arrow["d", shift left, from=1-3, to=1-4]
	\arrow["d", shift left, from=1-4, to=1-5]
	\arrow["\delta", shift left, from=1-5, to=1-4]
	\arrow["\delta", shift left, from=1-4, to=1-3]
	\arrow["\delta", shift left, from=1-3, to=1-2]
	\arrow["\delta", shift left, from=1-2, to=1-1]
\end{tikzcd}\]
\begin{definition}
   Let $V \in \mathrm{Ob}(\mathscr{C}_n)$ with decomposition $V = \bigoplus_i V_i$. The cohomology of $V$ regarded as the above complex in $d$ at the $i$-th position is denoted $H^i(V,d)$.
\end{definition}
\begin{definition}
    Let $V \in \mathrm{Ob}(\mathscr{C}_n)$ with decomposition $V = \bigoplus_i V_i$. The homology of $V$ regarded as the above complex in in $\delta$ at the $i$-th position is denoted $H_i(V,\delta)$.
\end{definition}
\begin{lemma}
Let $V \in \mathrm{Ob}(\mathscr{C}_n)$, and $v \in V$.
    If $v$ is closed, then $\star v$ is coclosed. Similarly, if $v$ is coclosed, then $\star v$ is is closed.
\end{lemma}
\begin{proof}
    Suppose that $dv=0$. Then we have that 
    \begin{align*}
        \delta (\star v) = \pm \star d\star \star v = \pm \star dv =0.
    \end{align*}
    Now, suppose $\delta v =0$, then we have that
    \begin{align*}
        \delta v &= 0 \\
        \pm \star d\star v &= 0 \\
        \pm d \star v &= \star0=0 \\
        \implies d(\star v) &= 0.
    \end{align*}
\end{proof}

\begin{lemma}
Let $V \in \mathrm{Ob}(\mathscr{C}_n)$, and $v \in V$.    If $v$ is exact, then $\star v$ is coexact. If $v$ is coexact, then $\star v$ is exact.
\end{lemma}
\begin{proof}
    Suppose that $v$ is exact, that is $v= dw$. Then we have that
    \begin{equation*}
        \star v = \star dw = \star d\star \star w = \pm \delta(\star w).
    \end{equation*}
    Similarly, suppose $v= \delta w$. Then we have 
    \begin{equation*}
        \star v = \star \delta w= \pm \star \star d\star w = d(\pm \star w).
    \end{equation*}
\end{proof}
\begin{theorem}
   For $V \in \mathrm{Ob}(\mathscr{C}_n)$, there is an isomorphism 
    \begin{equation*}
        H^i(\ker(\delta),d) \cong H^{-i}(\ker(d), \delta)
    \end{equation*}
    given by the $\star$ map.
\end{theorem}
\begin{proof}
    Define a map 
    \begin{equation*}
        \varphi:\widehat{V}_i/d(V_{i-1}\cap \ker(\delta)) \ra \widehat{V}_{-i}/\delta(V_{i+1}\cap \ker(d))
    \end{equation*}
    given by sending $[v] \mapsto [\star v]$. This is well-defined by the preceding lemmas, and is surjective as for any class $[v]$, it has preimage $[\star v]$. To see injectivity, suppose that $\varphi [v]$ vanishes, then by definition we have
    \begin{equation*}
        \star v = \delta w = \pm \star  d \star w
    \end{equation*}
    for some $w \in \ker(d) \cap V_{i+1}$, with sign depending on the grading ofthe argument. Applying $\star$ to the above relation we have
    \begin{equation*}
        v= \pm d \star w = d( \pm \star w)
    \end{equation*}
    which means that $[v]$ was the zero class. The only possible point of confusion is why $\pm \star w$ is in the kernel of $\delta$, but this again follows from the preceding lemmas.
\end{proof}
\subsection{Equivalence of Surjectivity}
The following theorem was established independently by Yan and Mathieu:
\begin{theorem}[Yan \cite{YAN1996143}, Mathieu \cite{Mathieu}]
    Let $(M^{2n},\omega)$ be a symplectic manifold. Then the following are equivalent:
    \begin{enumerate}
    \item The strong Lefschetz map $[\omega^k]:H^{n-k}_{dR}(M) \ra H^{n+k}_{dR}(M)$ is surjective.
    \item Every deRham cohomology class has has a (symplectic) harmonic representative.
\end{enumerate}
\end{theorem}
We can rephrase this purely as an algebraic equivalence: Let $V \in \mathrm{Ob}(\mathscr{C}_n)$, then the following are equivalent: 
\begin{enumerate}
        \item For all $0 \leq i \leq n$, the map $[e^i]:H^{-i}(V,d) \ra H^i(V,d)$ is surjection.
          \item The inclusion $(\ker(\delta),d) \hookrightarrow (V,d)$ induces a surjection in cohomology $H^i(\ker(\delta),d) \ra H^i(V,d)$.
    \end{enumerate}
\begin{lemma}\label{HarreptoHarrep}
    $[e^i]:H^{-i}(V,d) \ra H^i(V,d)$ carries harmonic representatives to harmonic representatives.
\end{lemma}
\begin{proof}
    Suppose $[v_{-i}] \in H^{-i}(V,d)$ has a harmonic representative $[v_{-i}^0]$. We claim that $[e^iv_{-i}^0]$ is a harmonic representative of $[e^iv_{-i}]$. First note that they're cohomologous, as if $v_{-i}=v_{-i}^0+d\tau_{-i-1}$, then we have that
    \begin{equation*}
        e^iv_{-i} = e^iv_{-i}^0 + e^id\tau_{-i-1} = e^iv_{-i}^0 + de^i\tau_{-i-1}.
    \end{equation*}
    Moreover, we have that $e^iv_{-i}^0$ is harmonic since
    \begin{equation*}
        \delta e^iv_{-i} = e^i \delta v_{-i} - ie^{i-1}dv_{-i} = 0.
    \end{equation*}
\end{proof}
The following is essentially a rephrasing of Yan's original proof \cite{YAN1996143}:
\begin{theorem}[D. Yan]\label{DoubleEquiv}
    Suppose $V \in \mathrm{Ob}(\mathscr{C}_n)$, then the following statements are equivalent:
    \begin{enumerate}
        \item The inclusion $(\ker(\delta),d) \hookrightarrow (V,d)$ induces a surjection in cohomology $H^i(\ker(\delta),d) \ra H^i(V,d)$.
        \item For all $0 \leq i \leq n$, the map $[e^i]:H^{-i}(V,d) \ra H^i(V,d)$ is a surjection.
    \end{enumerate}
\end{theorem}
\begin{proof} First we show that the first condition implies the second. Suppose condition $1$ holds, that is the maps $H^i(\ker(\delta),d) \ra H^i(V,d)$ are surjections, or equivalently we can say we have surjections $\widehat{V}_i \ra H^i(V,d)$. We also know by classical $\mathfrak{sl}_2$-representation theory that $e^i:V_{-i}\ra V_{i}$ is a bijection, and this restricts to a bijection on $\widehat{V}_{-i} \ra \widehat{V}_{i}$ since the harmonic elements make up a $\mathfrak{g}$-submodule of $V$. Together these give us the diagram
    % https://q.uiver.app/#q=WzAsNCxbMCwwLCJcXGtlcihcXGRlbHRhKSBcXGNhcCBWX3staX0iXSxbMSwwLCJcXGtlcihcXGRlbHRhKSBcXGNhcCBWX2kiXSxbMCwxLCJIXnstaX0oVixkKSJdLFsxLDEsIkheaShWLGQpIl0sWzAsMSwiW2VeaV0iXSxbMCwyXSxbMSwzXSxbMiwzLCJbZV5pXSIsMl1d
\[\begin{tikzcd}[ampersand replacement=\&]
	{ \widehat{V}_{-i}} \& {\widehat{V}_i} \\
	{H^{-i}(V,d)} \& {H^i(V,d)}
	\arrow["{e^i}", from=1-1, to=1-2]
	\arrow[from=1-1, to=2-1]
	\arrow[from=1-2, to=2-2]
	\arrow["{[e^i]}"', from=2-1, to=2-2]
\end{tikzcd}\]
where the top and vertical arrows are surjections, implying the bottom map is surjective. We now show that the second condition implies the first; towards this suppose that $[e^i]$ is surjective for all $i$. We'll show that given any cohomology class $[v] \in H^\bullet(V,d)$, we can replace it by a class $[w]$ where $w \in \ker(\delta)$. \newline
\textit{Step 1:}
First, we claim that it suffices to check this on cohomologies in nonpositive degree. To see this, fix $[v_i] \in H^i(V,d)$ for $i>0$. Then as $[e^i]$ as surjective, we have that $[v_i]=[e^i][w_{-i}]$ for some $[w_{-i}] \in H^{-i}(V,d)$, and we know by Propostion \ref{HarreptoHarrep} that $[e^i]$ will carry a harmonic representative of $[w_{-i}]$ to a harmonic representative of $[v_{i}]$. \newline
\textit{Step 2:}
Now, we claim that for $i>0$, $H^{-i}(V,d)$ decomposes as
\begin{equation*}
    H^{-i}(V,d) = \im([e]) + \ker([e^{i+1}]).
\end{equation*}
Fix a class $[v_{-i}] \in H^{-i}(V,d)$ and consider $[e^{i+1}v_{-i}] \in H^{i+2}(V,d)$. Since the $[e^i]$ are surjective, we have that there is class $[\beta_{-i-2}] \in H^{-i-2}(V,d)$ such that $[e^{i+2}\beta_{-i-2}]=[e^{i+1}v_{-i}]$. Now, we can (trivially) write
\begin{equation*}
    [v_{-i}] = ([v_{-i}-e\beta_{-i-2}])+[e]([\beta_{-i-2}]).
\end{equation*}
It's clear that the second summand is in the image of $[e]$, we need only to check that the first is in the kernel of $[e^{i+1}]$. This follows from the definition as 
\begin{equation*}
    [e^{i+1}]([v_{-i}-e\beta_{-i-2}]) = [e^{i+1}v_{-i}-e^{i+2}\beta_{-i-2}] = [e^{i+1}v_{-i} - e^{i+1}v_{-i}]=[0].
\end{equation*}
\newline
\textit{Step 3:}
We now proceed to prove the main claim by induction. First, note that every class in $H^{-n}(V,d)$ and $H^{-n+1}(V,d)$ is harmonic. If $[v_{-n}] \in H^{-n}(V,d)$, then $\delta v_{-n}\equiv 0$ trivially, and so $v_{-n}$ is harmonic. If $[v_{-n+1}] \in H^{-n+1}(V,d)$, then we have that
\begin{equation*}
    \delta v_{-n+1} = [d,f]v_{-n+1} = dfv_{-n+1} - fdv_{-n+1}=0
\end{equation*}
since $fv_{-n+1}\in V_{-n-1}\equiv 0$, and $v_{-n+1}$ is closed. So we only need to check on $-n+2 \leq i \leq 0$, and the $-n+1$ degree serves as the base case. Now, suppose the statement holds for $H^{-i-1}(V,d)$ and fix $[\beta_{-i}] \in H^{-i}(V,d)$. By the above decomposition we have that
\begin{equation*}
    [\beta_{-i}] = [\alpha_{-i}]+ [e]([\beta_{-i-2}]),
\end{equation*}
for $[\beta_{-i-2}]\in H^{-i-2}(V,d)$ and $[\alpha_{-i}] \in H^{-i}(V,d) \cap \ker([e^{i+1}])$. By inductive hypothesis, $[\beta_{-i-2}]$ possesses a harmonic representative, and $[e]$ carries harmonic representatives into harmonic representatives, so we can assume $[e]([\beta_{-i-2}])$ is harmonic. We only need to show that $[\alpha_{-i}]$ admits a harmonic representative. By definition of $[\alpha_{-i}]$, we have that $d\alpha_{-i}=0$ and $e^{i+1}\alpha_{-i}=d\beta_{i+1}$ for some $\beta_{i+1}\in V_{i+1}$. Now, by the bijectivity of $e^{i+1}:V_{-i-1}\ra V_{i+1}$, there exists $\gamma_{-i-1}\in V_{-i-1}$ such that $\beta_{i+1} = e^{i+1}\gamma_{-i-1}$. We claim that $\alpha_{-i}-d\gamma_{-i-1}$ is a harmonic representative of $[\alpha_{-i}]$. It's clear that $[\alpha_{-i}]=[\alpha_{-i}-d\gamma_{-i-1}]$, and that $\alpha_{-i}-d\gamma_{-i-1} \in \ker(d)$, we only need to show that it's in $\ker(\delta)$. Since $\delta=[f,d]$, this reduces to showing that $\alpha_{-i}-d\gamma_{-i-1} \in \ker(f)$, i.e. is primitive. Primitivity here is equivalent to showing that $\alpha_{-i}-d\gamma_{-i-1} \in \ker(e^{i+1})$. Computing we have that
\begin{align*}
    e^{i+1}(\alpha_{-i}-d\gamma_{-i-1}) = e^{i+1}\alpha_{-i} - d(e^{i+1}\gamma_{-i-1}) = e^{i+1}\alpha_{-i} - d\beta_{i+1} = e^{i+1}\alpha_{-i}-e^{i+1}\alpha_{-i}=0.
\end{align*}
\end{proof}

\subsection{Equivalence of Triples}
In the previous subsection we've shown that the following two statements are equivalent for $V \in \mathrm{Ob}(\mathscr{C}_n)$:
\begin{enumerate}
    \item The maps $[e^i]:H^{-i}(V,d) \ra H^i(V,d)$ are surjections.
    \item The inclusion $(\ker(\delta),d) \hookrightarrow (V,d)$ induces a surjection in cohomology.
\end{enumerate}
which in the original paper(s), Mathieu claims is equivalent to a third condition: The $d\delta$-lemma. 
\begin{definition}
    The $d \delta$-lemma holds for $V \in \mathrm{Ob}(\mathscr{C}_n)$ if the following equalities hold:
    \begin{equation*}
        \im(d) \cap \ker(\delta) = \im(\delta) \cap \ker(d)= \im (d \delta).
    \end{equation*}
\end{definition}
This is equivalent to the above two if we replace surjection with isomorphism in both claims; in the original setting of the tangent bundle of a compact manifold, these are equivalent by Poincaré duality. The first of the equivalences - that $d\delta$-lemma is equivalent to the inclusion $(\ker(\delta),d) \hookrightarrow (V, d)$ follows from the following:

\begin{theorem}\label{ddeltaequivrep}
    Let $V \in \mathrm{Ob}(\mathscr{C}_n)$. Then
    the $d\delta$-lemma is equivalent to the inclusion $(\ker(\delta),d) \ra (V, d)$ inducing an isomorphism in cohomology.
\end{theorem}
\begin{proof}
    See \cite[Lemma 5.4.1]{manin1998constructions}.
\end{proof}
\begin{prop}\label{ddeltaimpSLP}
    Let $V \in \mathrm{Ob}(\mathscr{C}_n)$, and suppose that the $d\delta$-lemma is satisfied. Then $[e^k]:H^{-k}(V,d) \ra H^k(V,d)$ is an isomorphism for all $k$.
\end{prop}
\begin{proof}
    Since $d\delta$-lemma implies that there are harmonic representatives, we have that that $[e^k]$ is surjective and we may reduce the problem to examining when the map $[e^k]:H^{-k}(\ker(\delta),d) \ra H^k(\ker(\delta),d)$ is an isomorphism. Suppose that $[\alpha_{-k}] \in \ker([e^k])$, so that $e^k\alpha_{-k}=d\tau_{k-1}$ for some $\tau_{k-1}\in V_{k-1} \cap \ker(\delta)$. Then, note that
    \begin{equation*}
        \delta(e^k(\alpha_{-k}))=\delta d \tau_{k-1}=-d\delta \tau_{k-1}=0,
    \end{equation*}
     thus $e^k\alpha_{-k} \in \im(d) \cap \ker(\delta)$. By the $d \delta$-lemma, we then have that there is a $\rho_k$ such that
    \begin{equation*}
        e^k\alpha_{-k} = d \delta \rho_k.
    \end{equation*}
    Since $e^k$ is a bijection for all $k$, note that there exists $\rho_{-k}$ such that $e^k\rho_{-k}=\rho_k$. Using this along with Proposition \ref{ddeltaebracket0},  we then have that
    \begin{align*}
        0 &= e^{k}\alpha_{-k} - d\delta \rho_k \\
         &=e^{k}\alpha_{-k} - d\delta e^k \rho_{-k} \\
         &= e^k(\alpha_{-k}-d\delta \rho_{-k}),
    \end{align*}
    thus by injectivity of $e^k$ we have that $\alpha_{-k}=d\delta \rho_{-k}$.
\end{proof}
\begin{prop}\label{KIisII}
    Let $V \in \mathrm{Ob}(\mathscr{C}_n)$. If $[e^i]:H^{-i}(V,d) \ra H^i(V,d)$ is an isomorphism for all $i$, then $\ker(d) \cap \im (\delta) = \im(d) \cap \im(\delta)$.
\end{prop}
\begin{proof}
    Unfolding the definitions reveals one must check that $d \delta v=0$ implies that $\delta v = d \phi$ for some $\phi$. First, we claim that it suffices to check this on primitive elements. Towards this, suppose that the statement holds on primitive elements, and fix arbitrary $v \in V$. We can decompose $v$ into primitive elements, i.e. $v =\sum e^i v_i$ where the $v_i$ are primitive. Now, since $d \delta v=0$, we have that
    \begin{equation*}
        0 = d \delta v = \sum d \delta e^i v_i  = \sum e^i d \delta v_i,
    \end{equation*}
    where we repeatedly use Proposition \ref{ddeltaebracket0}. Since this sum is direct, this forces $e^i d \delta v_i=0$ for all $i$. Since $e_i$ is acting on the negative weights in this sum, it acts injectively, so $d \delta v_i=0$ for all $i$. By hypothesis, we have then that there exist $\phi_i$ such that $\delta v_i = d \phi_i$, and thus
    \begin{align*}
        \delta v &= \sum \delta e^iv_i  \\
        &= \sum e^i \delta v_i - ie^{i-1}dv_i \\
        &= \sum e^i d \phi_i - i e^{i-1}dv_i \\
        &= d\left( \sum e^i \phi_i  - ie^{i-1}v_i\right).
    \end{align*}
    Now, to see that this holds on primitive form we proceed as follows. Suppose $v_{-i} \in V_{-i}$ is primitive and $d \delta v_{-i} =0$. Then we have that 
    \begin{equation*}
        e^{i+1}\delta v_{-i} = \delta e^{i+1}v_{-i} - (i+1)e^{i}dv_{-i}.
    \end{equation*}
    However, $e^{i+1}v_{-i}=0$ since $v_{-i}$ is primitive, and we're left with
    \begin{equation*}
        e^{i+1}\delta v_{-i} = d(-(i+1)e^iv_{-i}).
    \end{equation*}
    Rephrased this says that $[\delta v_{-i}] \in \ker [e^{i+1}]$, which by the strong Lefschetz property is an isomorphism, meaning that $[\delta v_{-i}]$ is the zero class in cohomology, so $\delta v_{-i}=d\phi$ for some $\phi \in V_{-i-2}$. The only degree where this argument does not hold is $V_{-n}$, but for any $v \in V_{-n}$, $\delta v \equiv 0$, so it's trivially in the image of $d$, i.e. $\delta v =0 = d(0)$.
\end{proof}

\begin{prop}\label{KIisIK} 
  Let $V \in \mathrm{Ob}(\mathscr{C}_n)$, and suppose that $[e^i]:H^{-i}(V,d) \ra H^i(V,d)$ is an isomorphism for all $i$. Then
    $\ker(d) \cap \im (\delta) = \ker(\delta) \cap \im(d)$.
\end{prop}
\begin{proof}
 From the preceding proposition, we have that $\ker(d) \cap \im(\delta)=\im(d) \cap \im(\delta) \subset \ker(\delta) \cap \im(d)$, all that remains to be shown is that $\im(d) \cap \ker( \delta) \subset \im(\delta) \cap \ker(d)$. Fix $\alpha \in \im(d) \cap \ker(\delta)$, that is $\alpha = d\beta$ and $\delta \alpha=0$. Then we have 
 \begin{equation*}
     0 = \delta \alpha = \pm \star d \star \alpha.
 \end{equation*}
 Since $\star$ is an isomorphism, we have that $d \star \alpha =0$,i.e. $\star \alpha \in \ker(d)$. Now using exactness we have
 \begin{equation*}
     \star \alpha = \star d \beta = \pm \star d \star \star \beta = \pm \delta \star \beta,
 \end{equation*}
 thus $\star \alpha \in \ker(d) \cap \im(\delta)$. By the preceding proposition yet again we have then that $\star \alpha \in \im(d) \cap \im (\delta)$, and so $\star \alpha = d \rho$. Using this we then have that
 \begin{equation*}
     \alpha = \star \star \alpha = \star d \rho = \star d \star \star \rho  =\delta (\pm \star \rho)
 \end{equation*}
 giving us $\alpha \in \im (\delta)$.
\end{proof}

\begin{prop}\label{SLPisddelta}
    Let $V \in \mathrm{Ob}(\mathscr{C}_n)$, and suppose that $[e]^i:H^{-i}(V,d) \ra H^i(V,d)$ is surjective for all $i$. Then $\im(d) \cap \im (\delta)= \im(d\delta)$.
\end{prop}
\begin{proof}
    We proceed by induction. First, note that the statement trivially holds elements in $V_n$ or $V_{-n}$ since if $v \in \im(\delta) \cap V_{n}$, then $v=\delta 0=0$, and likewise for $V_{-n}$. Our base case is on $V_{n-1}$, so choose $v_{n-1} \in V_{n-1}$, and let
    \begin{equation*}
        v_{n-1} = d\gamma_{n-2} = \delta w_n.
    \end{equation*}
    Since $dw_n$ is trivially zero it represents a cohomology class, and so we can choose a harmonic representative of said class giving us $w_n = w_n^0 - d \tau_{n-1}$ where $w_n^0$ is harmonic. We then have that
    \begin{align*}
        v_{n-1} = \delta w_n 
        = \delta( w_n^0 - d \tau_{n-1}) 
        =-\delta d \tau_{n-1} 
        = d \delta \tau_{n-1}.
    \end{align*}
    Now, suppose this holds for $V_k$ for some $-n<k<n$, and fix $v_{k-1}\in V_{k-1}\cap \im(d) \cap \im(\delta)$. By definition we have that 
    \begin{equation*}
    v_{k-1}=d \gamma_{k-2} = \delta \beta_k.
    \end{equation*}
    From this, define $\xi_{k+1} = d\beta_k$, and observe that
    \begin{equation*}
        \delta \xi_{k+1} = \delta d \beta_k = -d (\delta \beta_k) = -d^2\gamma_{k-2}=0,
    \end{equation*}
    so $\xi_{k+1}\in \im(d) \cap \ker(\delta)$. We have that $\im(d) \cap \ker(\delta) = \im(d) \cap \im (\delta)$, which by inductive hypothesis extends to $\im(d) \cap \ker(\delta) = \im(d) \cap \im (\delta) = \im (d\delta)$, so $\xi_{k+1}= d \delta \phi_{k+1}$. Note that $d(\beta_k - \delta \phi_{k+1})=0$, so it admits a harmonic representative, i.e. $\beta_k - \delta \phi_{k+1}=\beta_k^0 - d\tau_k$ for some $\tau_k \in V_k$. Applying the definition of $v_{k-1}$, we have 
    \begin{align*}
        v_{k-1} = \delta \beta_k 
        = \delta(\beta_k^0 - d\tau_k + \delta \phi_{k+1}) 
        = -\delta d \tau_k 
        = d \delta \tau_k.
    \end{align*}
\end{proof}
At this point, we have enough to establish the equivalence of the triple: $d\delta$-lemma, strong Lefschetz property, and Brylinski property. We've already proven all the of the relevant pieces, but for clarity state them here in one place.
\begin{theorem}\label{TripleEquiv}
    Let $V \in \mathrm{Ob}(\mathscr{C}_n)$. The the following conditions are equivalent:
    \begin{enumerate}
        \item The $d\delta$-lemma holds for $V$.
        \item The map $[e^i]:H^{-i}(V,d) \ra H^i(V,d)$ is an isomorphism for all $i$.
        \item The inclusion $(\ker(\delta),d) \hookrightarrow (V,d)$ is a quasi-isomorphism.
    \end{enumerate}
\end{theorem}
\begin{proof}
    We have that the first and third conditions are equivalent by Theorem \ref{ddeltaequivrep}. Then we have that the first implies the second by Proposition \ref{ddeltaimpSLP}, and the reverse direction follows from the Proposition \ref{KIisII}, \ref{KIisIK}, and \ref{SLPisddelta}.
\end{proof}
In light of this theorem, we make the following definition:
\begin{definition}
     If $V \in \mathrm{Ob}(\mathscr{C}_n)$ satisfies any of the above equivalent conditions in Theorem \ref{TripleEquiv} we call $V$ a ``Lefschetz Module''.
\end{definition}
\begin{corollary}\label{lefschetzgivessl2coho}
    Let $V \in \mathrm{Ob}(\mathscr{C}_n)$ be a Lefschetz module. Then $H^*(V,d)$ is an $\mathfrak{sl}_2(\R)$-module.
\end{corollary}
\begin{proof}
    From the $\mathfrak{g}$ relations we have that $e$ and $h$ always pass to cohomology, i.e. $H^*(V,d)$ is always a $B$-module, independent of whether or not $V$ is Lefschetz. To define the action of $f$, fix $[v] \in H^*(V,d)$, and let $\hat{v}$ be a harmonic representative of this class, which always exists as $V$ is Lefschetz. We can then define
    \begin{equation*}
        f[v] := [f(\hat{v})].
    \end{equation*}
    To see this is well-defined, suppose that $\hat{w}$ is another harmonic representative of the class, that is $\hat{w}=\hat{v}+d\phi$. Then we have that
    \begin{equation*}
        f(\hat{w}) = f(\hat{v}) + fd\phi = f(\hat{v}) + df\phi + \delta \phi,
    \end{equation*}
    where we use that $[f,d]=\delta$. Note that
    \begin{equation*}
        d\delta \phi = -\delta d \phi = -\delta ( \hat{w}-\hat{v})=0
    \end{equation*}
    by harmonicity, and so $ \delta \phi \in \im(\delta) \cap \ker(d)$. Again using that $V$ is Lefschetz, the $d\delta$-lemma implies that $\delta \phi = d\delta \varphi$ and so we have that
    \begin{equation*}
        f(\hat{w}) = f(\hat{v}) + d(f\phi + \delta \varphi).
    \end{equation*}
\end{proof}

\subsection{Weakly Lefschetz Modules}\label{Weak}
There are examples on manifolds where the Lefschetz condition is only partially met, i.e. the maps $[L]^k$ are  isomorphisms for some $k$, but not all $ 0 \leq k \leq n$.
\begin{definition}
    Fix $V \in \mathrm{Ob}(\mathscr{C}_n)$, $0 \leq s \leq n$, and consider the ordered set
    \begin{equation*}
        \{[e^n], [e^{n-1}], \dots [e^1] \},
    \end{equation*}
    where $[e^i]:H^{-i}(V,d) \ra H^i(V,d)$. We say that $V$ is ``$s$-Lefschetz'' if the first $s$ maps in this set are isomorphisms. In particular, $n$-Lefschetz is Lefschetz.
\end{definition}
The same theorems from before hold in this setting as well, but only on the portions on the modules for which the maps $e^k$ are isomorphisms. 
\begin{theorem}
      Suppose $V \in \mathrm{Ob}(\mathscr{C}_n)$ and fix $0 \leq s \leq n$. Then the following statements are equivalent:
    \begin{enumerate}
        \item The inclusion $(\ker(\delta),d) \hookrightarrow (V,d)$ induces a surjection in cohomology $H^i(\ker(\delta),d) \ra H^i(V,d)$ in degrees $i \in \{-n, \dots-s\} \cup \{s ,\dots, n\}$.
        \item The maps $[e^i]:H^{-i}(V,d) \ra H^i(V,d)$ for $i \in  \{n, \dots ,s\}$ are surjections.
    \end{enumerate}
\end{theorem}
\begin{proof}
    The forward direction follows in the same exact way it does in the proof of Theorem \ref{DoubleEquiv}. The reverse direction followed by induction that crucially used the fact that we could decompose $H^{-i}(V,d)$ for $i>0$. This decomposition was the only part of the proof that relied on the reverse direction, and the rest of the proof holds if we only assume that the $[e^i]$ are surjective from $V_{-n}$ up to $V_{-s}$.
\end{proof}
The upgraded equivalence also holds under the added assumption that $e^i$ is in fact an isomorphism.
\begin{definition}
    Given a module $V \in \mathrm{Ob}(\mathscr{C}_n)$ and $1 \leq s \leq n$, we say that $V$ has the ``$d\delta$-lemma up to degree $s$'' if 
    \begin{align*}
        \im(d) \cap \ker(\delta)&=\im (d \delta)= \im(\delta) \cap \ker(\delta) & \text{on } &V_i \text{ for } -n \leq i \leq -s \\
        \im(d) \cap \ker(\delta) &= \im(d\delta) & \text{on }& V_{-s+1}.
    \end{align*}
\end{definition}
\begin{remark}
    Note that by $\star$ duality if this equality holds on $V_i$ for $-n \leq i \leq -s$ it then also holds in $V_i$ for $s \leq i \leq n$. From this observation we'll only prove statements about it in non-positive degree, as the duality will extend the results to positive degrees.
\end{remark}

\begin{theorem} \label{WeakTheorem}
    Let $V \in \mathrm{Ob}(\mathscr{C}_n)$, and fix $1 \leq s \leq n$. Then the following are equivalent:
    \begin{enumerate}
        \item The $d\delta$-lemma holds for $V$ up to degree $s$.
        \item $V$ is $s$-Lefschetz
        \item The inclusion $(\ker(\delta),d) \hookrightarrow (V,d)$ is a quasi-isomorphism in degrees $i \in \{-n, \dots,-s\} \cup \{s ,\dots, n\}$.
    \end{enumerate}
\end{theorem}
\begin{proof}
The proof follows similarly to the original. The first and third condition are equivalent by \cite{manin1998constructions}, whose argument works in each degree. The first implies the second using Proposition \ref{ddeltaimpSLP} and acknowledging the proof is symmetric in it's degree arguments about the zero weight. Finally, to see that the second condition implies the first, note that proposition \ref{KIisII} and proposition \ref{KIisIK} are degree agnostic, and proposition \ref{SLPisddelta} is done by induction, which we can once again stop early when the hypothesis fail in degree $-s$ and achieve the result.
\end{proof}
\subsection{Primitive Equivalence}
Much of this section is rephrasing of work in Chung-I Ho's thesis \cite{HoSym} and Tseng's paper \cite{tseng2012cohomologyhodgetheorysymplectic}.

\begin{definition}
Suppose $V \in \mathrm{Ob}(\mathscr{C}_n)$ with decomposition $\bigoplus_k V_k$. Then the space of primitive elements of degree $k$ will be denoted  $\mathcal{P}_k = \ker(f) \cap V_k$.
\end{definition}
Throughout the rest of this subsection, fix $V \in \mathrm{Ob}(\mathscr{C}_n)$, so that we may speak of elements $v_i \in \Prim_i$ freely. Recall that for any $v_i \in \Prim_i$, we have that
\begin{equation*}
    dv_i = \dee v_i + e\deebar v_i.
\end{equation*}
These extend to the entire module by taking primitive decomposition of any element $v \in V$, i.e. if $v=\sum v_i$ for $v_i \in V_i$, and $v_i= \sum_{k \geq 0}e^k v_{i-2k}$ then we have that
\begin{align*}
    \dee v &= \sum_i\sum_{k \geq 0} e^k \dee v_{i-2k}, \\
     \deebar v &= \sum_i \sum_{k \geq 0} e^k \deebar v_{i-2k}.
\end{align*}

\begin{lemma}
    We have that $[\dee,e]=[\deebar,e]=[\dee,\dee]=[\deebar,\deebar]=0$ where all are graded commutators, as well as
    \begin{equation*}
        e(\dee \deebar) = - e (\deebar \dee)
    \end{equation*}
      on $V_{n-1},V_n$ and $[\dee,\deebar]=0$ on $V_k$ for $k < n-1$.
\end{lemma}
\begin{proof}
    This follows the same as it did in \cite[Lemma 2.5]{tseng2012cohomologyhodgetheorysymplectic}.
\end{proof}
\begin{prop}
    For $v_i \in \Prim_i$, we have that $f \dee v_i=0$.
\end{prop}
\begin{proof}
    Since $v_i$ is primitive, we have that $[f,\dee]v_i = f \dee v_i$. By the Jacobi identity we have that 
    \begin{equation*}
        [f,\dee] = \frac{1}{2}[h,[\dee ,f]] + [f, \dee],
    \end{equation*}
    thus we have to have that $h$ and $[\dee,f]$ commute. Checking this on $v_i$ gives us that
    \begin{equation*}
        (i-1)[\dee,f]v_i = i[\dee,f]v_i,
    \end{equation*}
    forcing $[f,\dee]v_i=0$.
\end{proof}
\begin{prop}
 For $v_i \in \Prim_i$ $\delta v_i = (1-i)\deebar v_i$.
\end{prop}
\begin{proof}
    Using the definition of $\delta$ we have
    \begin{equation*}
        \delta v_i = fdv_i - dfv_i = fdv_i = f\dee v_i + fe\deebar v_i.
    \end{equation*}
    We know the first summand vanishes by the preceding lemma, and on the second we have that
    \begin{equation*}
        fe\deebar v_i = (ef - h)\deebar v_i = -h\deebar v_i = -(i-1) \deebar v_i.
    \end{equation*}
\end{proof}
\begin{lemma}
    On $\Prim_i$, we have that $d \delta = (1-i) \dee \deebar$.
\end{lemma}
\begin{proof}
    Using the preceding lemmas we have
    \begin{equation*}
        d \delta v_i = (i-1)d \deebar v_i = (i-1)(\dee + e \deebar)\deebar v_i = (1-i)\dee \deebar v_i.
    \end{equation*}
\end{proof}
\begin{definition}
    We say that $v \in V$ is ``symphonic'' if $\dee v = \deebar v =0$.
\end{definition}
\begin{remark}
This was previously called ``primitive harmonic'' in the literature, see \cite{HoSym}, however to avoid the phrase ``harmonic primitive harmonic'' we introduce this terminology.
\end{remark}
\begin{lemma}
    $v \in V$ is harmonic if and only if each term of it's primitive decomposition is symphonic.
\end{lemma}
\begin{proof}
    Let $v = v_i$ with $v_i = \sum_{k \geq 0} e^k w_{i-2k}$ for $w_{i-2k} \in \Prim_{i-2k}$. Then if each $w_{i-2k}$ is symphonic, since $d =\dee  + e\deebar$ it's clear that $v$ is closed. To see $v$ is coclosed, observe
    \begin{align*}
        \delta v &= \sum_i \sum_{k \geq 0}\delta e^k w_{i-2k} \\
                 &= \sum_i \sum_{k \geq 0} (e^k\delta - ke^{k-1}d)w_{i-2k} \\
                 &= \sum_i \sum_{k \geq 0} e^k (1-i)\deebar w_{i-2k} \\
                 &=0.
    \end{align*}
    Now, suppose that $v =\sum_i v_i$ for $v_i \in V_i$ is harmonic. This occurs if and only if each $v_i$ is harmonic, which occurs if and only if the term in the primitive decomposition of each $v_i$ is harmonic, so it suffices to show this for a primitive harmonic element. Towards this, suppose $v \in \widehat{\Prim_i}$. Then
    \begin{equation*}
        0=\delta v = (i-1) \deebar v,
    \end{equation*}
    so $v \in \ker(\deebar)$. Moreover we have
    \begin{equation*}
        0 = dv = \dee v + e\deebar v = \dee v,
    \end{equation*}
    thus $v \in \ker( \dee)$.
\end{proof}
\begin{definition}
    Let $V \in \mathrm{Ob}(\mathscr{C}_n)$. We say that $V$ satisfies the $\dee \deebar$-lemma if the following equalities hold
    \begin{align*}
        \im(\dee) \cap \ker(\deebar) \cap \Prim_k &= \im(\dee\deebar) \cap \Prim_k \qquad -n+1 \leq k \leq -1, \\
        \im(\dee) \cap \ker(\deebar) \cap \Prim_{-n}&= \im(\dee \deebar)\cap \Prim_{-n}, \\
        \im(\deebar) \cap \ker(\dee) \cap \Prim_0 &= \im(\dee \deebar) \cap \Prim_0.
    \end{align*}
\end{definition}
\begin{lemma}\label{deedeebarimpliesddelta}
    The $d\delta$-lemma holds for $V \in \mathrm{Ob}(\mathscr{C}_n)$ if the $\dee\deebar$-lemma does.
\end{lemma}
\begin{proof}
    It suffices to check this on each weight, so fix $v_i \in \widehat{V_i}$ and first suppose that $v_i \in \im d$ so that $v_i = dw_{i-1}$. Then taking primitive decompositions of both we have
    \begin{equation*}
        v_i = \sum_{k \geq 0}e^kv_{i-2k},\qquad w_{i-1} = \sum_{k \geq 0}e^kw_{i-1-2k}.
    \end{equation*}
    Combining this with the fact that $v_i=dw_{i-1}$, we have
    \begin{equation*}
        \sum_{k \geq 0} e^k(v_{i-2k}-dw_{i-1-2k})=0,
    \end{equation*}
    forcing $v_{i-2k} = dw_{i-1-2k} = \dee w_{i-1-2k}+e\deebar w_{i-1-2k}$ for all $k$. Since $v_i$ is harmonic, each $v_{i-2k}$ is symphonic, so we must have that $\dee w_{i-1-2k} \in \im \dee \cap \ker \deebar$ and $\deebar w_{i-1-2k} \in \im \deebar \cap \ker \dee$. By the $\dee \deebar$-lemma then we must have that there are $\gamma_{i-2k}, \rho_{i-2-2k} \in \im(\dee \deebar)$ (possibly zero) such that
    \begin{equation*}
        \dee w_{i-1-2k} = \dee \deebar \gamma_{i-2k},  \qquad \deebar w_{i-1-2k} = \dee \deebar \rho_{i-2-2k}.
    \end{equation*}
   From here, we have that 
   \begin{align*}
       v_i &= \sum_{k \geq 0} e^k (\dee w_{i-1-2k}+e\deebar w_{i-1-2k}) \\
       &= \sum_{k \geq 0} e^k (\dee \deebar \gamma_{i-2k} + e\dee\deebar \rho_{i-2(1+k)}) \\
       &= \sum_{k \geq 0} e^k \left( d\delta \frac{\gamma_{i-2k}}{1-(i-2k)} + e d \delta \frac{\rho_{i-2(1+k)}}{1-i+2(1+k)} \right) \\
       &=d\delta \sum_{k \geq 0}e^k \left(  \frac{\gamma_{i-2k}}{1-(i-2k)} + e  \frac{\rho_{i-2(1+k)}}{1-i+2(1+k)} \right)
   \end{align*}
   where we repeatedly use the $\mathfrak{g}$ relations, and make note that the constants that appear in the denominator are non-vanishing. The only possible point of confusion is why the $\gamma_{i-2k}$ and $\rho_{i-2(1+k)}$ are primitive; if they're not taking a primitive decomposition of them once again allows the proof to go through. A similar argument holds for the case when $v \in \im(\delta)$.
\end{proof}
\begin{lemma}\label{ddeltaimpliesdeedeebar}
    The $\dee \deebar$-lemma holds for $V \in \mathrm{Ob}(\mathscr{C}_n)$ if the $d \delta$-lemma does.
\end{lemma}
\begin{proof}
    Suppose that $v_i \in \ker(\dee) \cap \ker(\deebar) \cap \Prim_i$ and recall that this forces $v_i$ to be harmonic and each of it's primitive decomposition terms to be symphonic. First, suppose that $v_i \in \im(\deebar)$, i.e. $v_i = \deebar \beta_{i+1}$. Taking primitive decompositions, we have
    \begin{equation*}
        v_i = \sum_{k \geq 0}e^k v_{i-2k}, \qquad \beta_{i+1}= \sum e^k\beta_{i+1-2k}.
    \end{equation*}
    Using coexactness, we have must have that
    \begin{equation*}
        v_{i-2k} = \deebar \beta_{i+1-2k} = \delta \left( \frac{\beta_{i+1-2k}}{-i+2k}\right)
    \end{equation*}
    so $v_{i-2k} \in \im (\delta) \cap \ker(d) = \im(d\delta)$. Thus there exists $\rho_{i-2k}$ such that
    \begin{equation*}
        v_{i-2k} = d \delta \rho_{i-2k} = \dee \deebar (1-i+2k)\rho_{i-2k},
    \end{equation*}
    giving us that
    \begin{equation*}
        v_i = \sum_{k \geq 0}e^k v_{i-2k} = \dee \deebar \sum_{k \geq 0} e^k((1-i+2k)\rho_{i-2k}).
    \end{equation*}
    Again it may be unclear as to why the $\rho_{i-2k}$ are primitive, but this can be chosen to happen as before. Now, suppose that $v_i \in \im(\dee)$ so $v_i = \dee \beta_{i-1}$ for some $\beta_{i-1}$. Equating primitive decompositions we have that $v_{i-2k}=\dee \beta_{i-1-2k}$, where
    \begin{equation*}
        v_i = \sum_{k \geq 0} e^kv_{i-2k}, \qquad \beta_{i-1} = \sum_{k \geq 0}e^k \beta_{i-1-2k}.
    \end{equation*}
    Since $v_i$ is symphonic, we have that $\deebar \dee \beta_{i-1-2k}=0$ for all $k$, and so
    \begin{align*}
        d \delta \beta_{i-1} &= \sum d \delta e^k \beta_{i-1-2k} \\
        &= \sum_{k \geq 0} e^k d \delta \beta_{i-1-2k} \\
        &= \sum_{k \geq 0} \frac{1}{2(1-k)-i} e^k \dee \deebar \beta_{i-1-2k} \\
        &=\sum_{k \geq 0} \frac{-1}{2(1-k)-i} e^k \deebar \dee \beta_{i-1-2k} \\
        &=0.
    \end{align*}
    Since $d \beta_{i-1} \in \ker(\delta) \cap \im(d)$, the $d \delta$-lemma then we have that there exist $\gamma_i$ such that $d\beta_{i-1} = d \delta \gamma_i$. Taking a primitive decomposition of $\gamma_i$ we have that
    \begin{align*}
        d \delta \gamma_i &= \sum_{k \geq 0} e^k d \delta \gamma_{i-2k} \\
        &=\sum_{k \geq 0} (1-i+2k)e^k \dee \deebar \gamma_{i-2k} \\
        &=\deebar \dee \sum_{k \geq 0} (i-1-2k)e^k\gamma_{i-2k},
    \end{align*}
    showing that
    \begin{equation*}
        v_i = \dee \beta_{i-1} - e\deebar \beta_{i-1} = \deebar \left( \dee \sum_{k \geq 0}(i-1-2k)e^k\gamma_{i-2k} - e\beta_{i-1} \right) \in \im(\deebar),
    \end{equation*}
    from which the result follows from the previous case.
\end{proof}
\begin{theorem}\label{QuadEquiv}
    Let $V \in \mathrm{Ob}(\mathscr{C}_n)$. Then the following are equivalent:
    \begin{enumerate}
        \item The $d\delta$-lemma holds for $V$.
        \item The map $[e^i]:H^{-i}(V,d) \ra H^i(V,d)$ is an isomorphism for all $i$.
        \item The inclusion $(\ker(\delta),d) \hookrightarrow (V,d)$ is a quasi-isomorphism.
        \item The $\dee \deebar$-lemma holds for $V$.
    \end{enumerate}
\end{theorem}
\begin{proof}
    The first three are equivalent by Theorem \ref{TripleEquiv}. The first and last are equivalent by Lemmas \ref{ddeltaimpliesdeedeebar} and \ref{deedeebarimpliesddelta}.
\end{proof}

\section{Lefschetz Algebroids}
\subsection{Symplectic Lie Algebroids}
Lie algebroids were originally introduced by Pradines \cite{zbMATH03247872} as a generalization of the tangent bundle of a manifold to the Lie Groupoid setting.
\begin{definition}
     A ``Lie algebroid'' $(\A, \rho, M)$ over a $C^\infty$-manifold $M$ is a vector bundle $\A \ra M$, equipped with an ``anchor'' $\rho:\A \ra TM$, and a bracket $[\cdot,\cdot]_\A$ on sections of $\A$ that is bilinear, alternating, and satisfies the Jacobi identity, such that
    \begin{enumerate}
        \item $[X,fY]_\A=f[X,Y]_\A + (\rho(X) \cdot f) Y$ for all $X,Y \in \Gamma(\A)$ and $f \in C^\infty(M)$.
        \item $\rho([X,Y]_\A) = [\rho(X),\rho(Y)]_{TM}$ for all $X,Y \in \Gamma(\A)$.
    \end{enumerate}
\end{definition}
\begin{definition}\label{DCEcomplex} Given a Lie algebroid $(\A, \rho,M)$, we can define the ``deRham-Chevalley-Eilenberg complex'' of $\A$, where our cochain complex is made of
\begin{equation*}
    \Omega^k_\A(M) = \Gamma(\Lambda^k\A^*),
\end{equation*}
and we define 
\begin{equation*}
    \Omega^*_\A(M) = \bigoplus_k \Omega^k_\A(M).
\end{equation*}
This comes with an $\A$-differential, $d_\A:\Omega^k_\A(M) \ra \Omega^{k+1}_\A(M)$, defined as
\begin{align*}
    d_\A \eta(X_1,...,X_{k+1}) &= \sum_{i=1}^{k+1} (-1)^{i+1}\rho(X_i)\cdot \eta (X_1,...,\widehat{X_i},...,X_{k+1}) \\
    &+ \sum_{i<j} (-1)^{i+j}\eta([X_i,X_j], X_1,...,\widehat{X_i},...,\widehat{X_j},...,X_{k+1})
\end{align*}
for $X_i \in \Gamma(\A)$ and $\eta \in \Omega^k_\A(M)$. 
\end{definition}
A standard computation shows that $d_\A^2=0$, and as such we can consider the cohomology of this complex:
\begin{definition}
    Let $(A,\rho,M)$ be a Lie algebroid. The ``Lie algebroid cohomology'' of $\A$, denoted $H^k_\A(M)$ is the cohomology of the complex
    % https://q.uiver.app/#q=WzAsNSxbMSwwLCJcXE9tZWdhXntrLTF9X1xcQShNKSJdLFsyLDAsIlxcT21lZ2Fea19cXEEoTSkiXSxbMywwLCJcXE9tZWdhXntrKzF9X1xcQShNKSJdLFswLDAsIlxcZG90cyJdLFs0LDAsIlxcZG90cyJdLFszLDAsImRfXFxBIl0sWzAsMSwiZF9cXEEiXSxbMSwyLCJkX1xcQSJdLFsyLDQsImRfXFxBIl1d
\[\begin{tikzcd}[ampersand replacement=\&]
	\dots \& {\Omega^{k-1}_\A(M)} \& {\Omega^k_\A(M)} \& {\Omega^{k+1}_\A(M)} \& \dots
	\arrow["{d_\A}", from=1-1, to=1-2]
	\arrow["{d_\A}", from=1-2, to=1-3]
	\arrow["{d_\A}", from=1-3, to=1-4]
	\arrow["{d_\A}", from=1-4, to=1-5]
\end{tikzcd}\]
that is,
\begin{equation*}
     H_\A^k(M) = \frac{\ker\left(\Omega^k_\A(M) \xrightarrow{d_\A}\Omega^{k+1}_\A(M)\right)}{\im \left( \Omega^{k-1}_\A(M) \xrightarrow{d_\A}\Omega^k_\A(M) \right)}
\end{equation*}
which we can rephrase by using $\A$-closed and exact terminology
\begin{equation*}
    H_\A^k(M) = \frac{d_\A\text{-closed forms}}{d_\A\text{-exact forms}}.
\end{equation*}
\end{definition}
The standard Cartan Calculus from the Tangent Bundle also extends to Lie algebroids, however we'll only use the notion of the exterior and interior products.
\begin{definition}
    Given $X \in \Gamma (\A)$, we define ``$\A$-interior multiplication'' with $X$, or ``contraction by $X$'', denoted $\iota_X$, as the operator $\iota_X:\Omega^k_\A(M) \ra \Omega^{k-1}_\A(X)$ such that
    \begin{equation*}
        \iota_X\eta(X_1,...,X_{k-1}) = \eta(X,X_1,...,X_{k-1}),
    \end{equation*}
    where $X_i \in \Gamma(\A)$ and $\eta \in \Omega_\A^k(M)$.
\end{definition}
We now arrive at the central object of the result, a symplectic Lie algebroid.
\begin{definition}
    Let $(\A, \rho , M)$ be a Lie algebroid. An ``$\A$-symplectic form'' $\omega$ is a $d_\A$-closed, nondegenerate element of $\Omega^2_\A(M)$.
\end{definition}
\begin{definition}
    A ``symplectic Lie algebroid'' is a Lie algebroid $\A$ equipped with a choice of $\A$-symplectic form.
\end{definition}
The presence of an $\A$-symplectic form immediately places restrictions on the underlying vector bundle:
\begin{corollary}
    Let $(\A \ra M, \rho ,\omega)$ be a symplectic Lie algebroid. Then the rank of $\A$ is even and $\A$ is orientable.
\end{corollary}
There are other obstructions as well, see \cite{klaasse2017geometric}. Symplectic Lie algebroids were originally introduced by Nest and Tsygan \cite{nest1999deformations} while investigating extensions of Fedesov Deformation Quantization, and much work has been done on them in recent years by may authors. As in classical symplectic geometry, symplectic Lie algebroids arise naturally as the ``phase space'' of any Lie algebroid as originally uncovered by Mart\'inez \cite{Martinez}.
\begin{prop}[Martinez]
    Let $B \ra N$ be an arbitrary Lie algebroid, and $\pi:B^* \ra N$ be the projection. Define the bundle
    \begin{equation*}
        \A = \pi^! B = TB^* \times_{TN}B
    \end{equation*}
    which is a Lie algebroid over $B^*$. Then $\omega = -d_\A\alpha$
    is a canonical symplectic structure on $B^*$, where $\alpha\in \Omega_\A^1(B^*)$ is defined as follows after identifying elements of $\A$ as tuples in the product bundle
    \begin{equation*}
        \alpha(x,p,v,b) = p(b) \qquad x \in N, p \in B^*_x , v \in T_p B^*, b \in B_x
    \end{equation*}
    satisfying $T_p \pi(v) = \rho_B(b)$. 
 \end{prop}
 \begin{proof}
     See originally Martinez \cite[Section 7]{Martinez} or Lin \cite[Proposition 2.2.1]{Lin_2023}.
 \end{proof}
 More recently, Lin et al. \cite{Lin_2023} proved an analogue of the Darboux-Marsden-Weinstein Theorem and described the processes of symplectic reduction in this context. In general, examples exist plentifully: 
\begin{itemize}
    \item In the case that the base manifold is a point, we recover the notion of a quasi-Frobenius Lie algebra. In this setting, a Frobienus Lie algebra corresponds to an $\A$-symplectic form that is exact, see \cite[Chapter 3.1]{MR1300632}.
    \item In the case that $\A=TM$, a $\A$-symplectic form is a symplectic form in the classical sense.
    \item $b$-symplectic manifolds, also known as log-symplectic manifolds give examples of symplectic Lie algebroids, as the $b$-tangent bundle, typically denoted ${}^bTM$ is a Lie algebroid with the inclusion map as an anchor, and the Lie bracket of vector fields as it's bracket. In a neighborhood of the hypersurface of interest, we have that any $b$-symplectic form looks like
    \begin{equation*}
        \omega = \frac{1}{x_1}dx_1 \wedge dx_2 +dx_3 \wedge dx_4 + \dots + dx_{2n-1} \wedge dx_{2n}, 
    \end{equation*}
    where $x_1$ is the direction along the hypersurface. For more on $b$-symplectic (and Poisson) geometry, see \cite{Guillemin_2014}.
    \item In celestial mechanics, the McGhee coordinate change was introduced \cite{McGehee1974} to help solve the planar reduced 3-body problem. In this setting, then natural symplectic structure becomes
    \begin{equation*}
        \omega = - \frac{4}{x^3}dx \wedge dy + d\alpha \wedge dG.
    \end{equation*}
    This is an example of $b^3$-symplectic manifold. In general, $b^k$-symplectic manifolds give examples of symplectic Lie algebroids. For more on these, see \cite{scott2013geometrybkmanifolds},\cite{Guillemin2017}.
    \item One often employs and encounters symplectic Lie algebroids in Poisson geometry. On some Poisson manifolds, one can construct a symplectic Lie algebroid which induces the underlying Poisson structure. In this case, the Poisson structure can be ``lifted'' to the algebroid, where it behaves much more like a symplectic object, and some computations are more amenable. Which Poisson manifolds admit these symplectic Lie algebroids is still an open question, though many families of examples are known. For more, see \cite{braddell2017invitationsingularsymplecticgeometry}.
\end{itemize}
If we have two symplectic Lie algebroids in hand, their product is also a symplectic Lie algebroid.
\begin{prop}
    Given symplectic Lie algebroids $(\A_1 \ra M_1,\omega_1)$ and $(\A_2 \ra M_2, \omega_2)$, the product $(A_1 \times A_2 \ra M_1 \times M_2 , \omega)$ where $\omega = \pi_1^* \omega_1 + \pi_2^*\omega_2$ is a symplectic Lie algebroid where $\pi_i:\A_1 \times \A_2 \ra \A_i$ are the projection maps.
\end{prop}
\begin{proof}
    The product of two Lie algebroids has a unique Lie algebroid structure, see \cite{GroupoidAlgebroidNotes}, so it remains to be seen that the above $\omega$ is indeed a symplectic form. $\omega$ is nondegenerate as it's made from the pullback of two nondegenerate forms. Note that $\pi_i$ are Lie algebroid morphisms \cite[Example 7.4]{GroupoidAlgebroidNotes}, and as such they induce chain maps $\pi_i^*$. Using this we have that
    \begin{equation*}
        d_{\A_1 \times \A_2}\omega = d_{\A_1 \times \A_2}\pi_1^* \omega_1 + d_{\A_1 \times \A_2} \pi_2^* \omega_2 = \pi_1^*(d_{\A_1}\omega_1) + \pi_2^*(d_{\A_2}\omega_2) = 0.
    \end{equation*}
\end{proof}

\subsection{Kahler-Weil Identities}
Throughout this section, let $(\A \ra M, \rho , \omega)$ be a symplectic Lie algebroid of rank $2m$. The ``Lefschetz'' map is defined on the space of $\A$ forms, and is explicitly defined as 
\begin{align*}
    L:\Omega^k_\A(M) &\ra \Omega^{k+2}_\A(M)  \\
    \alpha &\mapsto \omega \wedge \alpha.
\end{align*}
If we let $\pi$ denote the dual bivector to $\omega$, we can define another operator on the space of $\A$-forms:
\begin{align*}
    \Lambda:\Omega^{k}_\A(M) &\ra \Omega^{k-2}_\A(M) \\
     \alpha &\mapsto \iota_\pi \alpha.
\end{align*}
These two operators classically generate an $\mathfrak{sl}_2(\R)$ representation on the space of forms of a symplectic manifold, and the same is true in this setting as well. To prove this, we'll need to work locally however and we have the following weaker analogue of the Darboux Theorem from classical symplectic Geometry. This Lemma is essentially a rephrasing of \cite[Lemma 3.23]{garmendia2025estructuresregularpoissonmanifolds}, though we state it again here for containment sake.
\begin{lemma}\label{SLADarboux}
    Let $(\A \ra M, \rho ,\omega)$ be a symplectic Lie algebroid of rank $2m$. Then for all $p \in M$, there exists a smooth symplectic local frame of $\A$, i.e. we can locally express $\omega$ as
    \begin{equation*}
        \omega = \sum_{1 \leq i \leq m} \mu_i \wedge \xi_i
    \end{equation*}
    where $\mu_i,\xi_i \in \Gamma(\A^*|_U)$ for a sufficiently small open neighborhood $U$ of $p$.
\end{lemma}
\begin{proof}
    This is true for any symplectic vector bundle, see \cite[Lemma 12.6]{LibermannMarle1987}.
\end{proof}
\begin{remark}
    Note here that we're not insisting that the algebroid $1$-forms that make up our local coframe of $\A^*$ are differentials of coordinate functions, or that they're even closed. The local frame $\{\mu_i,\xi_i\}_{i=1}^m$ of $\A^*$ is dual to $\{u_i,v_i\}_{i=1}^m$, the local frame of $\A$ on the same open set, but a geometrically inclined choice of normal coordinates that correspond to these frames are harder to pin down. In fact, there are multiple obstructions to this. One of these is cohomological as shown by Miranda and Scott \cite[Theorem 29]{Miranda_2020}. Another is the failure for a general Lie algebroid to carry a commutative frame. This can be most easily in the case that our Lie algebroid is a Lie Algebra that is not abelian. For these reasons we will not make an appeal to a full ``Darboux Theorem'' though work has been done on this see \cite{Lin_2023}. In many settings such as $b$-manifolds there are results on local normal forms, and in these settings the following sections results can be more easily understood geometrically, see \cite{Matveeva_2023}, \cite{Miranda_2020}.
\end{remark}
Using the preceding lemma, at any point $p \in M$, we have local symplectic frames of $\A$ and $\A^*$ on an open set $U \ni p$ given by
\begin{equation*}
    \A|_U = \mathrm{span}_{C^\infty(M)}\{u_i,v_i\}_{i=1}^m,  \qquad
    \A^*|_U = \mathrm{span}_{C^\infty(M)}\{\mu_i,\xi_i\}_{i=1}^m.
\end{equation*}
In such a frame we can write local expressions for $L$ and $\Lambda$
\begin{equation*}
    L = \sum_i e_{\mu_i}e_{\xi_i} \qquad \Lambda =\sum_i \iota_{v_i}\iota_{u_i},
\end{equation*}
where $e_v$ denotes the map $v \wedge -$.Using these expressions, we can prove the following:
\begin{lemma}\label{KahlerWeil}
    If $\alpha \in \Omega^k_\A(M)$, then $[L,\Lambda]\alpha=(k-m)\alpha$.
\end{lemma}
\begin{proof}
    These are local, $C^\infty(M)$-linear operators, so fix a point $p \in M$ and an open neighborhood $U \ni p$ such that we have the symplectic frames $\{u_i,v_i\}_{i=1}^m$ of $\A$ and $\{\mu_i , \xi_i \}_{i=1}^m$ of $\A^*$ as provided by Lemma \ref{SLADarboux}. From here we can proceed as follows in \cite[Page 8]{Shlomonotes}.
\end{proof}

\begin{definition}
    We define the ``counting operator'' $H:\Omega^*_\A(M) \ra \Omega^*_\A(M)$ as the map
    \begin{equation*}
        H= \sum_k (k-m)\pi_k,
    \end{equation*}
    where $\pi_k:\Omega^*_\A(M) \ra \Omega^k_\A(M)$ is the canonical projection.
\end{definition}
\begin{prop}\label{AlgebroidKahlerWeil}
    Given the operators $L,\Lambda$, and $H$, we have the following $\mathfrak{sl}_2$-commutation relations
    \begin{equation*}
       [L,\Lambda]=H, \quad [H,L]=2L, \quad [H,\Lambda]=-2\Lambda.
    \end{equation*}
\end{prop}
\begin{proof}
    The first identity is the result of the lemma \ref{KahlerWeil}, and the remaining two follows from the following computations. Fix  $\alpha \in \Omega^k_\A(M)$, then we we have
    \begin{align*}
        [H,L]\alpha &= H(L\alpha) - L( (k-m)\alpha) \\
        &= (k+2-m)L\alpha - (k-m) L\alpha \\
        &= 2L\alpha,
    \end{align*}
    and
    \begin{align*}
        [H,\Lambda]\alpha &= H(\Lambda \alpha) - \Lambda (H\alpha) \\
        &=(k-2-m)\Lambda\alpha - (k-m)\Lambda \alpha \\
        &=-2\Lambda \alpha.
    \end{align*}
\end{proof}
These operators give the even part of the Lie Superalgebra $\mathfrak{g}$, and the Lie algebroid differential gives us the analogue of $d$, leaving us with one operator to define.
\begin{definition}
    The ``Koszul-Brylinski'' differential, denoted $d_\A^\star:\Omega^k_\A(M) \ra \Omega_\A^{k-1}(M)$, is defined as follows for any $\alpha \in \Omega^*_\A(M)$:
    \begin{equation*}
        d_\A^\star\alpha := [\Lambda, d_\A]\alpha.
    \end{equation*}
\end{definition}
This operator is of degree $-1$, and in the classical setting serves as the boundary map in \textit{Poisson Homology}. The fact $(d^\star_\A)^2=0$ is still unclear in our setting, but that will soon be remedied.
\begin{definition}
    We say that a form $\alpha \in \Omega^*_\A(M)$ is ``$\A$-symplectic harmonic'' or just ``harmonic'' if $d_\A \alpha=d_\A^\star\alpha=0$, and denote the set of harmonic $k$-forms as $\widehat{\Omega}_\A^k(M)$ and the set of all harmonic forms as $\widehat{\Omega}_\A^*(M)$. 
\end{definition}

\begin{prop}\label{CartanWedge2}
    Given $X,Y \in \Gamma(\wedge^2 \A)$, we have that $[\iota_X, [\iota_Y,d_\A]]=\iota_{[X,Y]}$ where $[\cdot,\cdot]$ is the Nijenhius-Schouten extension of the bracket on $\A$.
\end{prop}
\begin{proof}
    It suffices to check this locally on decomposable bivectors, so let $X=X_1 \wedge X_2$ and $Y=Y_1 \wedge Y_2$. We have that
    \begin{equation*}
        [\iota_Y,d_\A] =\iota_{Y_1}[\iota_{Y_2},d_\A]-[\iota_{Y_1},d_\A]\iota_{Y_2}
    \end{equation*}
    Applying $\iota_X$ and expanding, we have
    \begin{align*}
        [\iota_X,[\iota_Y,d_\A]] &= \underbrace{\iota_{X_1}[\iota_{X_2},\iota_{Y_1}[\iota_{Y_2},d_\A]]+[\iota_{X_1},\iota_{Y_1}[\iota_{Y_2},d_\A]]\iota_{X_2}}_{A} \\
        &-\underbrace{\left( \iota_{X_1}[\iota_{X_2},[\iota_{Y_1},d_\A]\iota_{Y_2}]+[\iota_{X_1},[\iota_{Y_1},d_\A]\iota_{Y_2}]\iota_{X_2}\right)}_B.
    \end{align*}
Note that using the cartan calculus identities $[\iota_a,\iota_b]=0$ and $[\iota_a,[\iota_b,d_\A]]=\iota_{[a,b]}$ for sections $a,b \in \Gamma(\A)$, we have that
\begin{equation*}
    [\iota_a,\iota_b[\iota_c,d_\A]]= \iota_b[\iota_a,[\iota_c,d_\A]]+[\iota_a,\iota_b][\iota_c,d_\A]= \iota_b \iota_{[a,c]}
\end{equation*}
Using this, we can simplify the above to
\begin{align*}
    A&=\iota_{X_1}\iota_{Y_1}\iota_{[X_2,Y_2]}+\iota_{Y_1}\iota_{[X_1,Y_2]}\iota_{X_2}, \\
    B&=\iota_{X_1}\iota_{[X_2,Y_1]}\iota_{Y_2} + \iota_{[X_1,Y_1]}\iota_{Y_2}\iota_{X_2}.
\end{align*}
Combining these and using that $\iota_a\iota_b = \iota_{a \wedge b}$, we get
\begin{align*}
    A-B &= \iota_{X_1\wedge Y_1 \wedge [X_2,Y_2]+Y_1 \wedge [X_1,Y_2]\wedge X_2 - X_1 \wedge [X_2,Y_1]\wedge Y_2 - [X_1,Y_1]\wedge Y_2 \wedge X_2} = \iota_{[X,Y]}
\end{align*}
\end{proof}
\begin{prop}\label{PoissonBracketVanish}
    Let $(\A \ra M,\rho,\omega)$ be a symplectic Lie algebroid, and let $\pi \in \Gamma(\bigwedge^2\A)$ be the dual bivector to $\omega$ as above. Then $[\pi,\pi]=0$.
\end{prop}
\begin{proof}
 See \cite[Proof of Theorem 2.14]{nest1999deformations} in the paragraph immediately following the list of 3 equivalences; there the bivector $\pi$ is written as $\varpi$.
\end{proof}

\begin{lemma}\label{AlgebroidCommutators}
    Let $(\A \ra M,\rho,\omega)$ be a symplectic Lie algebroid. Then we have the following commutation relations:
    \begin{enumerate}
        \item $[L,d_\A]=0$.
        \item $[H,d_\A]=d_A$.
        \item $[H,d_\A^\star]=-d_\A^\star$.
        \item $[\Lambda , d_\A^\star]=0$.
        \item $[L,d_\A^\star]=d_\A$.
        \item $[d_\A,d_\A^\star]=0$.
        \item $[d_\A^\star,d_\A^\star]=0$.
    \end{enumerate}
\end{lemma}
 \begin{proof}
        We prove these in order for $\alpha \in \Omega_{\A}^k(M)$:
        \begin{enumerate}
            \item Using that $\omega$ is closed we have
            \begin{align*}
                [L,d_\A]\alpha &= \omega \wedge d_\A\alpha - d_\A(\omega \wedge \alpha) \\
                &= \omega \wedge d_\A \alpha - (d_\A\omega \wedge \alpha  +\omega \wedge d_\A \alpha) \\
                &= 0.
            \end{align*}
            \item Using only degree considerations we have
            \begin{align*}
               [H,d_\A]\alpha = Hd_\A \alpha  -d_\A H\alpha = [(k+1)-n]d_\A \alpha - (k-n)d_\A \alpha = d_\A \alpha.
            \end{align*}
            \item Again using only degree considerations we have that
            \begin{equation*}
                [H,d_\A^\star]=Hd_\A^\star \alpha - d_\A^\star H \alpha = [(k-1)-n]d_\A^\star \alpha - (k-n)d_\A^\star \alpha = -d_\A^\star \alpha.
            \end{equation*}
            \item Using Proposition \ref{CartanWedge2}, we have that
            \begin{equation*}
               [\Lambda,d_\A^\star] =[\iota_\pi,[\iota_\pi,d_\A]]=\iota_{[\pi,\pi]}.
            \end{equation*}
            To show this vanishes, we must have that $[\pi,\pi]=0$, but this follows from Proposition \ref{PoissonBracketVanish}.
            \item Using the fact that $\omega_\A$ is $d_\A$-closed, and that $[L,d_\A]=0$, we have
            \begin{align*}
                [L,d_\A^\star]\alpha &= Ld_\A^\star \alpha - d_\A^\star L\alpha \\
                &=L\Lambda d_\A \alpha - Ld_\A \Lambda \alpha - \Lambda(\omega_\A \wedge d_\A \alpha) + d \Lambda (\omega_\A \wedge \alpha) \\
                &=[L,\Lambda]d_\A \alpha - d[L,\Lambda]\alpha \\
                &= [H,d_\A]\alpha \\
                &=d_\A\alpha.
            \end{align*}
            
            \item The graded Jacobi identity tells us that
            \begin{equation*}
                -[d_\A,[\Lambda,d_\A]]+[\Lambda,[d_\A,d_\A]]+ [d_\A,[d_\A,\Lambda]]=0.
            \end{equation*}
            From this we must have that
            \begin{equation*}
                -2[d_\A,d_\A^\star]=0.
            \end{equation*}
            
            \item Using items 4 and 6, along with the graded Jacobi identity we have that
            \begin{align*}
                [d_\A^\star,d_\A^\star]&=[d_\A^\star,[\Lambda, d_\A]] \\
                &=[\Lambda, [d_\A,d_\A^\star]]+[d_\A, [d_\A^\star, \Lambda]] \\
                &=0.
            \end{align*}
        \end{enumerate}
    \end{proof}
\begin{prop}
    The space of $\A$-forms of a symplectic Lie algebroid of rank $2m$ with the operators $L,\Lambda, H, d_\A, d_\A^\star$ is an object of $\mathscr{C}_m$ under the assignment
\begin{equation*}
    e \mapsto L, \quad f \mapsto \Lambda, \quad h \mapsto H, \quad d \mapsto d_\A, \quad \delta \mapsto d_\A^\star.
\end{equation*}
\end{prop}
\begin{proof}
    Ten of the twelve commutation relations follow Proposition \ref{AlgebroidKahlerWeil} and Lemma $\ref{AlgebroidCommutators}$. The last two come from the definition of $d_\A^*$, and the fact that $d_\A^2=0$.
\end{proof}

\begin{theorem}\label{AlgebroidQuad}
     Suppose $(\A \ra M,\rho, \omega)$ is a symplectic Lie algebroid of rank $2m$. Then the following statements are equivalent:
     \begin{enumerate}
         \item The inclusion $(\ker(d_\A^\star),d_\A) \hookrightarrow (\Omega_\A^*(M),d)$ is a quasi-isomorphism.
         \item For all $0 \leq k \leq m$, the map $[L]^{k}:H_\A^{m-k}(M) \ra H_\A^{m+k}(M)$ is an isomorphism.
         \item $(\A \ra M, \rho, \omega)$ satisfies the $\A$-$d d^\star$-lemma.
         \item $(\A \ra M, \rho, \omega)$ satisfies the $\A$-$\dee \deebar$-lemma.
     \end{enumerate}
     where the first three statements hold equivalence weakly, i.e. in degree as in the statement of Theorem $\ref{WeakTheorem}$.
\end{theorem}
\begin{proof}
    This is the statement of Theorem \ref{QuadEquiv} under the representation.
\end{proof}
\begin{definition}
    We say that a symplectic Lie algebroid $(\A,\rho, \omega)$ is a ``Lefschetz algebroid'' if any of the four conditions in the preceding theorem hold.
\end{definition}
\begin{corollary}
    For any Lefschetz algebroid $(\A \ra M, \rho, \omega)$, the space $H^*_\A(M)$ is an $\mathfrak{sl}_2(\R)$-module.
\end{corollary}
\begin{proof}
    This follows from Corollary \ref{lefschetzgivessl2coho}.
\end{proof}
\begin{corollary}
The deRham-Chevalley-Eilenberg complex of a Lefschetz algebroid is formal.
\end{corollary}
\begin{proof}
    Condition 1 in Theorem \ref{AlgebroidQuad} shows that a Lefschetz Lie algebroid satisfies the criteria for \cite[Lemma 5.4.1]{manin1998constructions}, from which the result follows.
\end{proof}
We also have the weak version of this theorem:
\begin{theorem}\label{AlgebroidWeak}
     Suppose $(\A \ra M,\rho, \omega)$ is a symplectic Lie algebroid of rank $2m$, and $0 \leq s \leq m$. Then the following statements are equivalent:
     \begin{enumerate}
         \item The inclusion $(\ker(d_\A^{\star}),d_\A) \hookrightarrow (\Omega_\A^\bullet(M),d)$ is a quasi-isomorphism in degrees $\leq s$ or $\geq s+m$.
         \item The map $[L]^k:H_\A^{m-k}(M) \ra H_\A^{m+k}(M)$ is an isomorphism for $k \geq s$.
         \item The $\A$-$d_\A d_\A^\star$-lemma holds up to degree $s$.
     \end{enumerate}
\end{theorem}
\begin{proof}
    This is the statement of Theorem \ref{WeakTheorem} under the representation.
\end{proof}
\begin{remark}
    An alternative method of constructing $d_\A^\star$ is to define the symplectic star $\star$ analogously to the classical setting, and then define $d_\A^\star:=(-1)^{k+1}\star d_\A \star$. In this case, the fact that $d_\A^\star = [\Lambda,d]$ becomes somewhat less clear to prove, as the classical method relies on tools that don't immediately generalize. Defining $d_\A^\star$ as the commutator allows us to show that the space of forms is indeed an object of $\mathscr{C}_m$, and there we constructed the $\star$ map without appeal to any underlying geometry, giving us both ways in a more straightforward manner.
\end{remark}
\subsection{Cohomology with Coefficients}
This whole story works as well for Lie algebroid cohomology with coefficients in a representation of the Lie algebroid. This extension was originally pointed out in \cite{YAN1996143} for traditional vector bundles with connections. 
\begin{definition}
   Let $\A \ra M$ be a Lie algebroid over a smooth manifold $M$. An ``$\A$-connection'' on a vector bundle $E \ra M$ is an $\R$-bilinear map
   \begin{align*}
       {}^\A\nabla:\Gamma \A \times \Gamma E &\ra \Gamma E \\
       (X ,s ) &\mapsto \nabla_X s
   \end{align*}
   which is $C^\infty(M)$ linear in the first entry and satisfies
   \begin{equation*}
       {}^\A\nabla_X (fs) = f {}^\A\nabla_X s + \mathcal{L}_{\rho(X)}(f) s   .
       \end{equation*}
     The ``curvature'' of the connection ${}^\A \nabla$ is the $\End(E)$-valued algebroid $2$-form
     \begin{equation*}
         R_\nabla(X , Y) = {}^\A \nabla_X {}^\A \nabla_Y - {}^\A \nabla_Y {}^\A \nabla _X - {}^\A \nabla_{[X,Y]}.
     \end{equation*}
     We say that ${}^\A \nabla$ is ``flat'' if $ R_ \nabla \equiv 0$.
\end{definition}
\begin{definition}
    Let $\A \ra M$ be a Lie algebroid over a smooth manifold. A ``representation'' of $\A$ is a vector bundle $E$ equipped with a flat $\A$-connection.
\end{definition}

\begin{definition}
    Let $\A \ra M$ be a Lie algebroid and $E$ be a representation of $\A$. Then the space of ``Algebroid $k$-forms with coefficients in $E$'' is defined as
    \begin{equation*}
        \Omega_\A^k(M,E) = \Gamma(\bigwedge^k \A^* \otimes E),
    \end{equation*}
    which assemble into differential graded module:
    \begin{equation*}
        \Omega_\A^*(M,E) = \bigoplus_{k \geq 0} \Omega^k_\A(M,E).
    \end{equation*}
    This caries a differential ${}^\A \nabla:\Omega_\A^k(M,E) \ra \Omega_\A^{k+1}(M,E)$ given by the induced covariant derivative, i.e.  we extend ${}^\A \nabla:\Omega_\A^0(M,E) \ra \Omega_\A^1(M,E)$ via enforcing
\begin{equation*}
    {}^\A\nabla (\alpha \otimes s) :=(d_\A \alpha) \otimes s + (-1)^{|\alpha|}\alpha \otimes ({}^\A \nabla s),
\end{equation*}
for $\alpha \otimes s \in \Omega^k_A(M,E)$. 
Since the connection is flat, this is a cochain complex.
\end{definition}

\begin{prop}
    Let $(\A, \rho, \omega)$ be a symplectic Lie algebroid of rank $2m$ and $E$ be a representation of $\A$. Define operators:
    \begin{align*}
    L_\nabla:\Omega^k_\A(M,E) &\ra \Omega_\A^{k+2}(M,E) \\
    \alpha \otimes e &\mapsto (\omega \wedge \alpha) \otimes e = (L\alpha)\otimes e, \\
    \Lambda_\nabla :\Omega_\A^k(M,E) &\ra \Omega_\A^{k-2}(M,E) \\
    \alpha \otimes e &\mapsto (\iota_{\omega^{-1}}\alpha) \otimes e = (\Lambda \alpha) \otimes e,\\
    H_\nabla : \Omega_\A^k(M,E) &\ra \Omega_\A^k(M,E) \\
    \alpha \otimes e &\mapsto ((k-m)\alpha )\otimes e  = (H \alpha) \otimes e,\\
    {}^\A \nabla^*:\Omega_\A^k(M,E) &\ra \Omega_\A^{k-1}(M,E) \\
    \alpha \otimes e &\mapsto [\Lambda_\nabla, {}^\A \nabla] (\alpha \otimes e).
    \end{align*}
These, together with ${}^\A \nabla$, satisfy the relations of $\mathfrak{g}$, and as such the space of forms with coefficients are an object of $\mathscr{C}_n$.
\end{prop}
\begin{proof}
    The proof follows similarly to that of Lemma \ref{AlgebroidCommutators}.
\end{proof}
\begin{definition}
    Let $(\A, \rho , M)$ be a symplectic Lie algebroid of rank $2m$, $E \ra M$ be a representation of $\A$ with connection ${}^\A \nabla$, and $\alpha \in \Omega^*_\A(M,E)$. We say that $\alpha$ is ``${}^\A \nabla$-symplectic Harmonic'' if ${}^\A \nabla \alpha = {}^\A \nabla^* \alpha =0$.
\end{definition}
\begin{corollary}
    Let $(\A \ra M, \rho, \omega)$ be a symplectic Lie algebroid of rank $2m$, with representation $(E,\nabla)$. Then the following are equivalent:
    \begin{enumerate}
        \item The map $[L]^k:H^{m-k}_\A(M,E) \ra H^{m+k}_\A(M,E)$ given by $[\omega]^k \wedge -$ is an isomorphism for all $k$.
        \item The inclusion of complexes $(\ker({}^\A\nabla^*), {}^\A \nabla) \hookrightarrow (\Omega_\A^\bullet(M,E),{}^\A \nabla)$ is a quasi-isomorphism.
        \item The $(\A,E)$-$d\delta$-lemma holds for $\A$.
        \item The $(\A,E)$-$\dee \deebar$-lemma holds for $\A$.
    \end{enumerate}
     and the first three statements hold equivalence weakly, i.e. in degree as in the statement of Theorem $\ref{WeakTheorem}$.
\end{corollary}
The above corollary is the statement of Theorems \ref{WeakTheorem} and \ref{QuadEquiv}. Taking $E$ to be the trivial line bundle recovers Theorem \ref{AlgebroidQuad}.

\begin{remark}
    It would be of interest to explore this result on the line bundle
\begin{equation*}
    E = Q_\A =  \bigwedge^{\text{top}}\A^* \otimes \bigwedge^{\text{top}}T^*M.
\end{equation*}
This was originally considered in \cite{evens1996transversemeasuresmodularclass}, where they showed it is always a representation of $\A$.
\end{remark} 

\section{Examples}
\subsection{Kodaria-Thurston Manifold}
 One of the first examples of a symplectic manifold that does not admit a K\"ahler structure is given by the ``Kodaira-Thurston'' manifold $KT^4$ which is constructed as follows:
Take $\R^4$ with coordinates $(x_1,x_2,x_3,x_4)$ and choose $a,b,c,d \in \Z$. Then consider the quotient of $\R^4$ by identifying
    \begin{equation*}
        (x_1,x_2,x_3,x_4) \sim (x_1+a,x_2+b,x_3+c, x_4+d-bx_3).
    \end{equation*}
    The resulting quotient is the connected smooth manifold $KT^4$. The chosen basis of one forms $\{e_i\}_{i=1}^4 $ is given by
    \begin{equation*}
    e_1 = dx_1 \qquad e_2=dx_2 \qquad e_3 = dx_3 \qquad e_4 = dx_4 + x_2dx_3.
    \end{equation*}
    We have a symplectic structure on $KT^4$ given by $ \omega = e_1 \wedge e_2 + e_3 \wedge e_4$, and the deRham cohomology of $KT^4$ has the following generators:
    \begin{align*}
    H_{dR}^0(KT^4) &= \langle [1] \rangle, & H_{dR}^1(KT^4) &= \langle [e_1], [e_2], [e_3] \rangle, \\
    H_{dR}^2(KT^4) &= \langle \omega , [e_{12} - e_{34}] , [e_{13}], [e_{24}] \rangle, & H_{dR}^3(KT^4) &=\langle [\omega \wedge e_1] ,[\omega \wedge e_2],[\omega \wedge e_4]\rangle, \\
    H_{dR}^4(KT^4)&= \langle [\frac{1}{2}\omega^2] \rangle.
\end{align*}
The two Lefschetz maps of interest are:
\begin{align*}
    [L]:H_{dR}^1(KT^4) \ra H_{dR}^3(KT^4)  \qquad  [L]^2:H_{dR}^0(KT^4) \ra H_{dR}^4(KT^4).
\end{align*}
The map $[L]^2$ is clearly an isomorphism, but $[L]$ is not as it has nontrivial kernel. This satisfies the criteria of Theorem \ref{AlgebroidQuad} of degree $s=1$ where we take $\A=TM$, the identity as the anchor, and the differential as the classical deRham differential.
\subsection{Symplectic Lie Algebras of Dimension Four}
If we consider symplectic Lie algebras as symplectic Lie algebroids over the one-point space, the Lie algebroid cohomology is exactly the Chevalley-Eilenberg cohomology, and our theorem can be checked by examining the surjectivity of the maps $[L]^k:H^{m-k}(\mathfrak{g}) \ra H^{m+k}(\mathfrak{g})$. For clarification sake, by a symplectic Lie algebra we mean a (real) Lie Algebra $\mathfrak{g}$ equipped with a nondegenerate $\omega \in \bigwedge^2 \mathfrak{g}^*$ that is closed under the Chevalley-Elienberg differential. These are also referred to as \textit{quasi-Frobenius Lie algebras}. The 2-dimensional case has only two isomorphism classes; the abelian one which satisfies Theorem \ref{AlgebroidWeak} and the other one which does not. Things get more interesting in the 4-dimensional case, where the number of symplectomorphism classes grows significantly. In these case we need to check the following two maps in cohomology in the diagram:
% https://q.uiver.app/#q=WzAsNyxbMCwwLCIwIl0sWzEsMCwiXFxMYW1iZGFeMCBcXG1hdGhmcmFre2d9XioiXSxbMiwwLCJcXExhbWJkYV4xIFxcbWF0aGZyYWt7Z31eKiJdLFszLDAsIlxcTGFtYmRhXjIgXFxtYXRoZnJha3tnfV4qIl0sWzQsMCwiXFxMYW1iZGFeMyBcXG1hdGhmcmFre2d9XioiXSxbNSwwLCJcXExhbWJkYV40IFxcbWF0aGZyYWt7Z31eKiJdLFs2LDAsIjAiXSxbMCwxXSxbMSwyXSxbMiwzXSxbMyw0XSxbNCw1XSxbNSw2XSxbMiw0LCJMXjEiLDEseyJjdXJ2ZSI6LTJ9XSxbMSw1LCJMXzIiLDEseyJjdXJ2ZSI6LTR9XV0=
\[\begin{tikzcd}[ampersand replacement=\&]
	0 \& {H^0 (\mathfrak{g})} \& {H^1 (\mathfrak{g})} \& {H^2( \mathfrak{g})} \& {H^3(\mathfrak{g})} \& {H^4( \mathfrak{g})} \& 0
	\arrow[from=1-1, to=1-2]
	\arrow[from=1-2, to=1-3]
	\arrow[from=1-3, to=1-4]
	\arrow[from=1-4, to=1-5]
	\arrow[from=1-5, to=1-6]
	\arrow[from=1-6, to=1-7]
	\arrow["{[L]^1}"{description}, bend left=30, from=1-3, to=1-5]
	\arrow["{[L]^2}"{description}, bend left=30, from=1-2, to=1-6]
\end{tikzcd}\]
 A result by Ovando \cite{ovando2004dimensional} gives a classification of all 4-dimensional real solvable symplectic Lie algebras up to symplectomorphism. Recomputing the cohomology of these classes we can check by hand which maps $[L]^i$ are isomorphisms, surjections, or neither. Using the notation of Ovando and numbering for different symplectic forms, we can check the maps in cohomology which results in the following table where $s$ denotes the $s$ as in Theorem \ref{AlgebroidWeak}, i.e. $s=2$ is the Lefschetz condition. 

\begin{table}[h]
\centering
\begin{minipage}{0.48\textwidth}
\centering
\begin{tabular}{|c|c|c|c|}
\hline
Lie Algebra & $s$ & $[L]$ & $[L]^2$ \\ \hline
$\mathfrak{rh}_3$ & 1 & Neither & Isomorphism\\ \hline
$\mathfrak{rr}_{3,0}$ &0  & Surjection & Surjection\\ \hline
$\mathfrak{rr}_{3,-1}$ & 2  & Isomorphism & Isomorphism \\ \hline
$\mathfrak{rr}_{3,0}'$ & 2 & Isomorphism & Isomorphism\\ \hline
$\mathfrak{r}_2\mathfrak{r}_2$ & 0  & Surjection & Surjection\\ \hline
$\mathfrak{r}_2'$ & 1 & 0 & Isomorphism\\ \hline
$\mathfrak{n}_4$ & 1 & 0 & Isomorphism \\ \hline
$\mathfrak{r}_{4,0}$ & 0 & Surjection & Surjection \\ \hline
$\mathfrak{r}_{4,-1}$ & 0 & 0 & Surjection\\ \hline
$\mathfrak{r}_{4,-1,\beta}, \beta \in (-1,0)$ & 0  & Isomorphism & Surjection \\ \hline
\end{tabular}
\end{minipage}\hfill%
\begin{minipage}{0.48\textwidth}
\centering
\begin{tabular}{|c|c|c|c|}
\hline
Lie Algebra & $s$ & $[L]$ & $[L]^2$ \\ \hline
$\mathfrak{r}_{4,-1,-1}$ & 0 & Neither & Surjection\\ \hline
$\mathfrak{r}_{4,\alpha, - \alpha}, \alpha \in (-1,0)$&0 & Surjection & Surjection  \\ \hline
$\mathfrak{r}_{4,0,\delta}', \delta >0$ &0  & Isomorphism & Surjection \\ \hline
$\mathfrak{d}_{4,1}$ & 0 & Surjection & Surjection\\ \hline
$\mathfrak{d}_{4,2}$, $\omega_1$ & 0  & 0 & Surjection\\ \hline
$\mathfrak{d}_{4,2}$, $\omega_2$ & 0  & Isomorphism & Surjection\\ \hline
$\mathfrak{d}_{4,2}$, $\omega_3$ & 0  & Isomorphism & Surjection \\ \hline
$\mathfrak{d}_{4,\lambda}, \lambda \geq 1/2$ & 0 & Surjection & Surjection \\ \hline
$\mathfrak{d}_{4,\delta}', \delta >0$ & 0 & Surjection & Surjection \\ \hline
$\mathfrak{h}_4$ & 0  & Surjection & Surjection \\ \hline
\end{tabular}
\end{minipage}

\caption{Lie Algebra Lefschetz Maps}
\end{table}

\subsection{Six-Dimensional Nilmanifolds}
Nilmanifolds are homogeneous spaces $N/\Gamma$ where $N$ is a simply connected nilpotent real Lie Group, and $\Gamma$ is a lattice in $N$ of maximal rank. Such spaces have been classified up to isomorphism in dimension $6$, and there are $34$ such isomorphism classes. It is known that exactly 26 of the 34 isomorphism classes admit symplectic structures, see \cite{NilpLieAlg}. The data of such a space is typically presented by two pieces of information: Its class, given as a $6$-tuple, and its symplectic form. For example, the tuple $(0,0,0,0,12,14+25)$ with symplectic form $13+26+45$ conveys the data of the Lie Algebra with dual generators $\{e^i\}_{i=1}^6$ such that
\begin{align*}
    de^1 = de^2&=de^3=de^4=0, \\
    de^5 &= e^1 \wedge e^2, \\
    de^6 &= e^1\wedge e^4 + e^2 \wedge e^5.
\end{align*}
with symplectic form $\omega = e^1 \wedge e^3 + e^2 \wedge e^6 + e^4 \wedge e^5$. Only one of these classes admit a K\"ahler structure, and the other classes have provided interesting (counter)-examples in symplectic geometry. The fact that the Lie algebra data is sufficient to classify these follows from a result of Nomizu.
\begin{theorem*}[Nomizu \cite{MR64057}]
   The deRham complex $\Omega^\bullet(M)$ of a nilmanifold $M=N/\Gamma$ is quasi-isomorphic to the complex $\bigwedge^\bullet \mathfrak{n}^*$ of left-invariant forms on $N$, thus $H_{dR}^\bullet(M) \cong H_{CE}^\bullet(\mathfrak{n})$.
\end{theorem*}
Examining the maps $[L]^k:H^{3-k}(\mathfrak{n}) \ra H^{3+k}(\mathfrak{n})$ where $[\alpha] \mapsto [\omega^k \wedge \alpha ]$ for $k=1,2,3$ allows to check when Theorem \ref{AlgebroidWeak} holds.
The first of these maps $[L]^3$ is assured to be an isomorphism by compactness of the nilmanifold, and the remaining two are explicitly calculable given the standard data of a symplectic nilmanifold described above. By Poincare\'e Duality, the maps $[L]^k$ are between vector spaces of the same finite dimension, and so being surjective is equivalent to being an isomorphism. Only one class will admit all three isomorphisms, as in this setting all three maps being isomorphisms is equivalent to admitting K\"ahler structure \cite{MR993739}. Table \ref{NilTable} lists whether or not a given nilmanifold class has maps $[L]^k$ that are isomorphisms; Yes indicating it's an isomorphism, and No indicating it's not an isomorphism.

\begin{longtable}{|c|c|c|c|c|}
\caption{Six-Dimensional Nilmanifold Lefschetz Maps}
\label{NilTable}\\

\hline
Signature  & Symplectic Form & $[L]$ & $[L]^2$ & $[L]^3$\\
\hline
\endfirsthead

\hline
Signature  & Symplectic Form & $[L]$ & $[L]^2$ & $[L]^3$\\
\hline
\endhead

\hline
\multicolumn{5}{r}{\textit{Continued on next page}} \\
\endfoot

\hline
\endlastfoot

(0,0,12,13,14,15) & 16+34-25 & No & No & Yes\\ \hline
(0,0,12,13,14,23+15) & 16+24+34-26 & No & No& Yes\\ \hline
(0,0,12,13,23,14) & 15+24+34-26 & Yes & No & Yes\\ \hline
(0,0,12,13,23,14-25) & 15+24-35+16  & Yes & No & Yes\\ \hline
(0,0,12,13,23,14+25)  & 15+24+35+16  & Yes  & No & Yes\\ \hline
(0,0,12,13,14+23,24+15)  & $16+2\times 34-25$  & No & No& Yes \\ \hline
(0,0,0,12,13,14+23)  &$16-2\times 34 -25$  & No & No& Yes \\ \hline
(0,0,0,12,13,24) & 26+14+35  & No & No & Yes\\ \hline
(0,0,0,12,13,14)  & 16+24+35 & No & No & Yes\\ \hline
(0,0,0,12,13,23) & 15+24+36 & No & No & Yes\\ \hline
(0,0,0,12,14,15+23) & 13+26-45 & No & No & Yes\\ \hline
(0,0,0,12,14+15+23+24) & 13+26-45 & No & No & Yes\\ \hline 
(0,0,0,12,14,15+24) & 13+26-45 & No & No & Yes\\ \hline
(0,0,0,12,14,15) & 13+26-45 & No & No & Yes\\ \hline
(0,0,0,12,14,13+42) & 15+26+34 & No & No & Yes\\ \hline
(0,0,0,12,14,23+24) & 16-34+25 & No & No & Yes\\ \hline 
(0,0,0,12,14,15+34) & 16+35+24 & No & No & Yes\\ \hline
(0,0,0,12,14+23,13+42) & $15+2 \times 26+34$ & No & No & Yes\\ \hline
(0,0,0,0,12,15) & 16+25+34 & No & No & Yes\\ \hline
(0,0,0,0,12,14+25) & 13+26+45 & No & No & Yes\\ \hline
(0,0,0,0,12,14+23) & 3+26+45 & No & No & Yes\\ \hline
(0,0,0,0,12,34) & 15+36+24 & No & No & Yes\\ \hline 
(0,0,0,0,12,13) & 16+25+34 & No & No & Yes\\ \hline
(0,0,0,0,13+42, 14+23) & 16+25+34 & No & No & Yes\\ \hline
(0,0,0,0,0,12) & 16+23+45 & No & No & Yes\\ \hline
(0,0,0,0,0,0) & 12+34+56 & Yes & Yes & Yes \\ \hline

\end{longtable}

\subsection{\texorpdfstring{$E$}{}-manifolds}
$E$-manifolds are a generalization of a collection of different geometric settings: $b$-manifolds, $b^m$-manifolds, $c$-manifolds, and Regular Foliations. Miranda and Scott have thoroughly explored the cohomology and symplectic geometry of $E$-manifolds in many settings, e.g. \cite{Miranda_2020}. The definition given in the above paper of Miranda and Scott of an $E$-manifold is restated here for convenience.
\begin{definition}
Let $E$ be a locally free submodule of the $C^{\infty}$ module $\mathfrak{X}(M)$ of vector fields on $M$. By the Serre-Swan theorem, there is an $E$-tangent bundle ${ }^E T M$ whose sections (locally) are sections of $E$, and an $E$-cotangent bundle ${ }^E T^* M:=\left({ }^E T M\right)^*$. We will call the global sections of $\Lambda^p\left({ }^E T^* M\right)$ $E$-forms of degree $p$, and denote the space of all such sections by ${ }^E \Omega^p(M)$. If $E$ satisfies the involutivity condition $[E, E] \subseteq E$, there is a differential $d:{ }^E \Omega^p(M) \rightarrow{ }^E \Omega^{p+1}(M)$ given by
$$
\begin{aligned}
{}^E d \eta\left(V_0, \ldots, V_p\right)= & \sum_i(-1)^i V_i\left(\eta\left(V_0, \ldots, \hat{V}_i, \ldots, V_p\right)\right) \\
& +\sum_{i<j}(-1)^{i+j} \eta\left(\left[V_i, V_j\right], V_0, \ldots, \hat{V}_i, \ldots, \hat{V}_j, \ldots, V_p\right) .
\end{aligned}
$$
\end{definition}
The cohomology of this complex is known as $E$-cohomology and is the same as Lie algebroid cohomology if we view this as a Lie Algebroid with anchor given by inclusion. If we have an $E$-closed nondegenerate $E$-$2$-form $\omega$, we call this an $E$-symplectic form and the triple $(E,M,\omega)$ an $E$-symplectic manifold. These are symplectic Lie Algebroids. The cohomology of such an object in general is rather difficult to compute, however Miranda and Scott showed the following:
\begin{theorem}[Miranda and Scott \cite{Miranda_2020}]
    Let $M=\R^2$, and let $E$ be the involutive subbundle of $\mathfrak{X}(\R^2)$ generated by
    \begin{align*}
        v_1 &= x \p_x + y \p_y, \quad
        v_2 = -y \p_x + x \p y.
    \end{align*}
    with dual bundle generated by
    \begin{align*}
        v_1^* &= \frac{xdx + ydy}{x^2+y^2}, \quad
        v_2^* = \frac{-ydx+xdy}{x^2+y^2}.
    \end{align*}
    Then the $E$-cohomology or Lie algebroid cohomology is given by
    \begin{equation*}
        {}^EH^k(M) = \begin{cases}
            \R \quad &i=0,2 \\
            \R^2 &i=1 \\
            0 &i\geq 3
        \end{cases}
    \end{equation*}
\end{theorem}

\begin{theorem}
    The symplectic Lie algebroid defined by Miranda and Scott $(E \ra \R^2, \rho = \id, \omega = v_1^* \wedge v_2^*)$ satisfies Theorem \ref{AlgebroidQuad}.
\end{theorem}
\begin{proof}
    We only need to check that the map $[L]:{}^EH^0(\R^2) \ra {}^EH^2(\R^2)$ is an isomorphism, which amounts to showing that $\omega = v_1^* \wedge v_2^*$ is not an exact form. Suppose towards a contradiction that it is, i.e. there are $f,g \in C^\infty(\R^2)$ such that $\omega = {}^Ed(fv_1^* + g v_2^*)$. Unfolding the definition this implies that
    \begin{equation*}
        \omega = (v_1(g) - v_2(f))\omega
    \end{equation*}
    which means for all $(x,y) \in \R^2$ we have that 
    \begin{equation*}
        v_1(g)-v_2(f)=1.
    \end{equation*}
    Equivalently
    \begin{equation*}
        \left(\frac{\p g}{\p x} - \frac{\p f }{\p y}\right)x+ \left( \frac{\p g}{\p x} + \frac{\p f}{\p x} \right)y=1,
    \end{equation*}
    however at $(0,0)$ this fails, so $\omega$ is not exact.
\end{proof}
\begin{remark}
    The above work does not restrict to just $E$-manifolds though. Theorem \ref{AlgebroidQuad} applies to the setting of symplectic $b^m$-manifolds for any $m \in \Z_{>0}$ as well, since the $b^m$-tangent bundle in these cases is a symplectic Lie algebroid. For more on these see \cite{Guillemin_2014}. 
\end{remark}

\subsection{K\"ahler Lie Algebroids}
There has been recent attention towards K\"ahler Lie Algebroids, see \cite{hu2024kodairavanishingtheoremskahler}. By the Hard Lefschetz Theorem, K\"alher manifolds are always Lefschetz manifolds. This relationship still holds true for a suitable class of K\"ahler Lie algebroids:
\begin{definition}
    A ``K\"ahler Lie algebroid'' $(\A \ra M, \rho,g)$ is a Lie algebroid with a Hermitian metric $g$ whose associated canonical $(1,1)$-$\A$-form $\omega \in \Omega_\A^{1,1}(M)$ is closed.
\end{definition}
\begin{corollary}
    K\"ahler Lie algebroids are symplectic Lie algebroids.
\end{corollary}
Throughout this section, let $J$ be the associated complex structure for $\A$. Reconsidering the construction of $d_\A^\star$ with the presence of a complex structure reveals that
\begin{equation*}
    J d_\A^\star J^{-1} = d_\A^\dagger 
\end{equation*}
where $d_\A^\dagger$ is the metric adjoint to $d_\A$, for details see \cite{Shlomonotes}.
\begin{theorem}
    K\"ahler Lie algebroids which admit harmonic representatives for every $\A$-cohomology class in the Riemannian sense are Lefschetz algebroids.
\end{theorem}
\begin{proof}
    It suffices to show that every cohomology class admits a symplectic harmonic representative. Fix any class $[\alpha] \in H_\A^{p,q}(M,\C)$, where we've chosen $\alpha$ to be a harmonic representative. Then $d_A \alpha=d_\A^\dagger \alpha =0$, and so
    \begin{equation*}
       0= d_\A^\dagger \alpha = Jd_\A^\star J^{-1}\alpha.
    \end{equation*}
    This forces $d_\A^\star \alpha =0$ up to a nonzero constant, so $\alpha$ is symplectic harmonic.
\end{proof}
\begin{remark}
    When a Kähler Lie algebroid admits a harmonic representative in the Riemannian sense is not immediately obvious, as the complex of $\A$-forms is not always elliptic.
\end{remark}

\printbibliography

@article {merkulov1998formality,
    AUTHOR = {Merkulov, S. A.},
     TITLE = {Formality of canonical symplectic complexes and Frobenius manifolds},
   JOURNAL = {Internat. Math. Res. Notices},
  FJOURNAL = {International Mathematics Research Notices},
      YEAR = {1998},
    NUMBER = {14},
     PAGES = {727--733},
      ISSN = {1073-7928,1687-0247},
   MRCLASS = {58F05 (16E99 58A12 58H15)},
  MRNUMBER = {1637093},
MRREVIEWER = {J.\ Terilla},
       DOI = {10.1155/S1073792898000439},
       URL = {https://doi.org/10.1155/S1073792898000439},
}

@article {manin1998constructions,
    AUTHOR = {Manin, Yu.\ I.},
     TITLE = {Three constructions of {F}robenius manifolds: a comparative study},
      NOTE = {Sir Michael Atiyah: a great mathematician of the twentieth century},
   JOURNAL = {Asian J. Math.},
  FJOURNAL = {Asian Journal of Mathematics},
    VOLUME = {3},
      YEAR = {1999},
    NUMBER = {1},
     PAGES = {179--220},
      ISSN = {1093--6106,1945--0036},
   MRCLASS = {14N35 (14J32 32S70 53D45)},
  MRNUMBER = {1701927},
MRREVIEWER = {Andreas\ Gathmann},
       DOI = {10.4310/AJM.1999.v3.n1.a8},
       URL = {https://doi.org/10.4310/AJM.1999.v3.n1.a8},
}

@article {nest1999deformations,
    AUTHOR = {Nest, Ryszard and Tsygan, Boris},
     TITLE = {Deformations of symplectic Lie algebroids, deformations of holomorphic symplectic structures, and index theorems},
   JOURNAL = {Asian J. Math.},
  FJOURNAL = {Asian Journal of Mathematics},
    VOLUME = {5},
      YEAR = {2001},
    NUMBER = {4},
     PAGES = {599--635},
      ISSN = {1093--6106,1945--0036},
   MRCLASS = {53D55 (32G05 58H15 58J22)},
  MRNUMBER = {1913813},
MRREVIEWER = {Martin\ Schlichenmaier},
       DOI = {10.4310/AJM.2001.v5.n4.a2},
       URL = {https://doi.org/10.4310/AJM.2001.v5.n4.a2},
}

@article {Martinez,
    AUTHOR = {Mart\'inez, Eduardo},
     TITLE = {Lagrangian mechanics on {L}ie algebroids},
   JOURNAL = {Acta Appl. Math.},
  FJOURNAL = {Acta Applicandae Mathematicae},
    VOLUME = {67},
      YEAR = {2001},
    NUMBER = {3},
     PAGES = {295--320},
      ISSN = {0167-8019,1572-9036},
   MRCLASS = {37J99 (53D05 58H05 70G45 70H03 70S05 70S10)},
  MRNUMBER = {1861135},
MRREVIEWER = {Frans\ Cantrijn},
       DOI = {10.1023/A:1011965919259},
       URL = {https://doi.org/10.1023/A:1011965919259},
}

@misc{klaasse2017geometric,
      title={Geometric Structures and Lie Algebroids}, 
      author={Ralph L. Klaasse},
      year={2017},
      eprint={1712.09560},
      archivePrefix={arXiv},
      primaryClass={math.SG},
}

@article {YAN1996143,
    AUTHOR = {Yan, Dong},
     TITLE = {Hodge structure on symplectic manifolds},
   JOURNAL = {Adv. Math.},
  FJOURNAL = {Advances in Mathematics},
    VOLUME = {120},
      YEAR = {1996},
    NUMBER = {1},
     PAGES = {143--154},
      ISSN = {0001--8708,1090--2082},
  MRNUMBER = {1392276},
MRREVIEWER = {Keizo Hasegawa},
       DOI = {10.1006/aima.1996.0034},
       URL = {https://doi.org/10.1006/aima.1996.0034},
}

@article{BrylinskiPoisson,
author = {Jean--Luc Brylinski},
title = {A Differential Complex for Poisson Manifolds},
volume = {28},
journal = {Journal of Differential Geometry},
number = {1},
publisher = {Lehigh University},
pages = {93 -- 114},
year = {1988},
doi = {10.4310/jdg/1214442161},
URL = {https://doi.org/10.4310/jdg/1214442161},
}

@article{Mathieu,
	author = {Mathieu, Olivier},
	doi = {10.1007/BF02565997},
	id = {Mathieu1995},
	journal = {Commentarii Mathematici Helvetici},
	number = {1},
	pages = {1--9},
	title = {Harmonic Cohomology Classes of Symplectic Manifolds},
	url = {https://doi.org/10.1007/BF02565997},
	volume = {70},
	year = {1995},
}

@article {zbMATH03247872,
    AUTHOR = {Pradines, Jean},
     TITLE = {Th\'eorie de {L}ie pour les groupo\"ides diff\'erentiables.
              {C}alcul diff\'erenetiel dans la cat\'egorie des groupo\"ides
              infinit\'esimaux},
   JOURNAL = {C. R. Acad. Sci. Paris S\'er. A-B},
  FJOURNAL = {Comptes Rendus Hebdomadaires des S\'eances de l'Acad\'emie des
              Sciences. S\'eries A et B},
    VOLUME = {264},
      YEAR = {1967},
     PAGES = {A245--A248},
      ISSN = {0151-0509},
   MRCLASS = {53.42},
  MRNUMBER = {216409},
MRREVIEWER = {K.-T.\ Chen},
}

@article {Lin_2023,
    AUTHOR = {Lin, Yi and Loizides, Yiannis and Sjamaar, Reyer and Song, Yanli},
     TITLE = {Symplectic reduction and a {D}arboux-{M}oser-{W}einstein theorem for {L}ie algebroids},
   JOURNAL = {Pure Appl. Math. Q.},
  FJOURNAL = {Pure and Applied Mathematics Quarterly},
    VOLUME = {19},
      YEAR = {2023},
    NUMBER = {4},
     PAGES = {2067--2131},
      ISSN = {1558-8599,1558-8602},
   MRCLASS = {53D17 (53D20)},
  MRNUMBER = {4671391},
       DOI = {10.4310/pamq.2023.v19.n4.a13},
       URL = {https://doi.org/10.4310/pamq.2023.v19.n4.a13},
}

@article {equivariantddlemma,
    AUTHOR = {Lin, Yi and Sjamaar, Reyer},
     TITLE = {Equivariant symplectic {H}odge theory and the {$d_G\delta$}-lemma},
   JOURNAL = {J. Symplectic Geom.},
  FJOURNAL = {The Journal of Symplectic Geometry},
    VOLUME = {2},
      YEAR = {2004},
    NUMBER = {2},
     PAGES = {267--278},
      ISSN = {1527-5256,1540-2347},
   MRCLASS = {53D20 (58A12)},
  MRNUMBER = {2108377},
MRREVIEWER = {Catalin\ Zara},
       URL = {http://projecteuclid.org/euclid.jsg/1094072007},
}

@misc{Guilleminddlemma,
    author = {Guillemin, Victor},
    title = {Symplectic Hodge Theory and the {$d\delta$}-lemma} ,
    journal = {Massachusets Institute of Technology},
    year = {2001}
}

@misc{Shlomonotes,
    author = {Guillemin, Victor},
    title = {Hodge Theory},
    journal = {Massachusets Institute of Technology},
    url = {https://math.mit.edu/~vwg/shlomo-notes.pdf},
    year = {1997},
    note = {Available at \url{https://math.mit.edu/~vwg/shlomo-notes.pdf}}
}

@book{GroupoidAlgebroidNotes,
    author = {Meinrenken, Eckhard},
    title = {Lie Groupoids and Lie Algebroids Lecture Notes, Fall 2017},
    publisher ={University of Toronto} ,
    year = {2017}
}

@article {ovando2004dimensional,
    AUTHOR = {Ovando, Gabriela},
     TITLE = {Four dimensional symplectic {L}ie algebras},
   JOURNAL = {Beitr\"age Algebra Geom.},
  FJOURNAL = {Beitr\"age zur Algebra und Geometrie. Contributions to Algebra and Geometry},
    VOLUME = {47},
      YEAR = {2006},
    NUMBER = {2},
     PAGES = {419--434},
      ISSN = {0138-4821},
   MRCLASS = {53D05 (17B56)},
  MRNUMBER = {2307912},
MRREVIEWER = {Anna\ M.\ Fino},
}

@article {Miranda_2020,
    AUTHOR = {Miranda, Eva and Scott, Geoffrey},
     TITLE = {The geometry of {$E$}-manifolds},
   JOURNAL = {Rev. Mat. Iberoam.},
  FJOURNAL = {Revista Matem\'atica Iberoamericana},
    VOLUME = {37},
      YEAR = {2021},
    NUMBER = {3},
     PAGES = {1207--1224},
      ISSN = {0213-2230,2235-0616},
   MRCLASS = {53D17 (57R30)},
  MRNUMBER = {4236806},
MRREVIEWER = {Tomasz\ Rybicki},
       DOI = {10.4171/rmi/1232},
       URL = {https://doi.org/10.4171/rmi/1232},
}

@book {HoSym,
    AUTHOR = {Ho, Chung-I},
     TITLE = {Topological {M}ethods in {S}ymplectic {G}eometry},
      NOTE = {Thesis (Ph.D.)--University of Minnesota},
 PUBLISHER = {ProQuest LLC, Ann Arbor, MI},
      YEAR = {2011},
     PAGES = {100},
      ISBN = {978-1124-80857-4},
   MRCLASS = {99-05},
  MRNUMBER = {2912247},
       URL = {http://gateway.proquest.com/openurl?url_ver=Z39.88-2004&rft_val_fmt=info:ofi/fmt:kev:mtx:dissertation&res_dat=xri:pqdiss&rft_dat=xri:pqdiss:3466948},
}

@article {Matveeva_2023,
    AUTHOR = {Matveeva, Anastasia and Miranda, Eva},
     TITLE = {Reduction theory for singular symplectic manifolds and singular forms on moduli spaces},
   JOURNAL = {Adv. Math.},
  FJOURNAL = {Advances in Mathematics},
    VOLUME = {428},
      YEAR = {2023},
     PAGES = {Paper No. 109161, 42},
      ISSN = {0001-8708,1090-2082},
   MRCLASS = {53D17},
  MRNUMBER = {4603782},
       DOI = {10.1016/j.aim.2023.109161},
       URL = {https://doi.org/10.1016/j.aim.2023.109161},
}

@article {tseng2012cohomologyhodgetheorysymplectic,
    AUTHOR = {Tseng, Li-Sheng and Yau, Shing-Tung},
     TITLE = {Cohomology and {H}odge theory on symplectic manifolds: {II}},
   JOURNAL = {J. Differential Geom.},
  FJOURNAL = {Journal of Differential Geometry},
    VOLUME = {91},
      YEAR = {2012},
    NUMBER = {3},
     PAGES = {417--443},
      ISSN = {0022-040X,1945-743X},
   MRCLASS = {53D40 (58A14)},
  MRNUMBER = {2981844},
MRREVIEWER = {Frederik\ Witt},
       URL = {http://projecteuclid.org/euclid.jdg/1349292671},
}

@book {hu2024kodairavanishingtheoremskahler,
    AUTHOR = {Hu, Tengzhou},
     TITLE = {A {G}eneralized {K}odaira {V}anishing {T}heorem for {L}ie
              {A}lgebroids and a {R}iemann-{R}och {T}heorem for {S}ingular
              {F}oliation},
      NOTE = {Thesis (Ph.D.)--Washington University in St. Louis},
 PUBLISHER = {ProQuest LLC, Ann Arbor, MI},
      YEAR = {2023},
     PAGES = {254},
      ISBN = {979-8381-08987-5},
   MRCLASS = {99-05},
  MRNUMBER = {4722602},
       URL = {https://gateway.proquest.com/openurl?url_ver=Z39.88-2004&rft_val_fmt=info:ofi/fmt:kev:mtx:dissertation&res_dat=xri:pqm&rft_dat=xri:pqdiss:30692294},
}

@article {evens1996transversemeasuresmodularclass,
    AUTHOR = {Evens, Sam and Lu, Jiang-Hua and Weinstein, Alan},
     TITLE = {Transverse measures, the modular class and a cohomology
              pairing for {L}ie algebroids},
   JOURNAL = {Quart. J. Math. Oxford Ser. (2)},
  FJOURNAL = {The Quarterly Journal of Mathematics. Oxford. Second Series},
    VOLUME = {50},
      YEAR = {1999},
    NUMBER = {200},
     PAGES = {417--436},
      ISSN = {0033-5606,1464-3847},
   MRCLASS = {53D17 (58H05)},
  MRNUMBER = {1726784},
MRREVIEWER = {Janez\ Mr\v cun},
       DOI = {10.1093/qjmath/50.200.417},
       URL = {https://doi.org/10.1093/qjmath/50.200.417},
}

@misc{fernández2005ddeltalemmaweaklylefschetzsymplectic,
      title={The {$d\delta$}--lemma for weakly Lefschetz symplectic manifolds}, 
      author={Marisa Fernández and Vicente Muñoz and Luis Ugarte},
      year={2005},
      eprint={math/0501259},
      archivePrefix={arXiv},
      primaryClass={math.SG},
      url={https://arxiv.org/abs/math/0501259}, 
}

@incollection {MR1480723,
    AUTHOR = {Mathieu, Olivier},
     TITLE = {Homologies associated with {P}oisson structures},
 BOOKTITLE = {Deformation theory and symplectic geometry ({A}scona, 1996)},
    SERIES = {Math. Phys. Stud.},
    VOLUME = {20},
     PAGES = {177--199},
 PUBLISHER = {Kluwer Acad. Publ., Dordrecht},
      YEAR = {1997},
      ISBN = {0--7923--4525--8},
   MRCLASS = {58F05 (58F06 58H10)},
  MRNUMBER = {1480723},
MRREVIEWER = {Manuel\ de Le\'on},
}

@book{NilpLieAlg,
  title     = "Nilpotent Lie Algebras",
  author    = "Goze, Michel and Khakimdjanov, Yusupdjan ",
  year      = 1996,
  publisher = "Springer",
  ISBN = {978-0-7923-3932-8},
  address   = "London"
}

@article {MR64057,
    AUTHOR = {Nomizu, Katsumi},
     TITLE = {On the cohomology of compact homogeneous spaces of nilpotent
              {L}ie groups},
   JOURNAL = {Ann. of Math. (2)},
  FJOURNAL = {Annals of Mathematics. Second Series},
    VOLUME = {59},
      YEAR = {1954},
     PAGES = {531--538},
      ISSN = {0003-486X},
   MRCLASS = {20.0X},
  MRNUMBER = {64057},
MRREVIEWER = {H.\ Freudenthal},
       DOI = {10.2307/1969716},
       URL = {https://doi.org/10.2307/1969716},
}

@article {MR993739,
    AUTHOR = {Benson, Chal and Gordon, Carolyn S.},
     TITLE = {K\"ahler structures on compact solvmanifolds},
   JOURNAL = {Proc. Amer. Math. Soc.},
  FJOURNAL = {Proceedings of the American Mathematical Society},
    VOLUME = {108},
      YEAR = {1990},
    NUMBER = {4},
     PAGES = {971--980},
      ISSN = {0002-9939,1088-6826},
   MRCLASS = {53C55 (22E25 22E40 32M05 32M10)},
  MRNUMBER = {993739},
MRREVIEWER = {Yusuke\ Sakane},
       DOI = {10.2307/2047955},
       URL = {https://doi.org/10.2307/2047955},
}

@misc{garmendia2025estructuresregularpoissonmanifolds,
      title={E-structures and almost regular Poisson manifolds}, 
      author={Alfonso Garmendia and Eva Miranda},
      year={2025},
      eprint={2410.11641},
      archivePrefix={arXiv},
      primaryClass={math.SG},
      url={https://arxiv.org/abs/2410.11641}, 
}

@article{Guillemin_2014,
   title={Symplectic and Poisson geometry on b-manifolds},
   volume={264},
   ISSN={0001-8708},
   url={http://dx.doi.org/10.1016/j.aim.2014.07.032},
   DOI={10.1016/j.aim.2014.07.032},
   journal={Advances in Mathematics},
   publisher={Elsevier BV},
   author={Guillemin, Victor and Miranda, Eva and Pires, Ana Rita},
   year={2014},
   month=oct, pages={864–896} }

@book{LibermannMarle1987,
  author       = {Paulette Libermann and Charles-Michel Marle},
  title        = {Symplectic Geometry and Analytical Mechanics},
  series       = {Mathematics and Its Applications},
  volume       = {35},
  publisher    = {D. Reidel / Springer},
  address      = {Dordrecht},
  year         = {1987},
  isbn         = {90-277-2438-5},
  doi          = {10.1007/978-94-009-3807-6},
  url          = {https://link.springer.com/book/10.1007/978-94-009-3807-6}
}

@book {MR2906817,
    AUTHOR = {Musson, Ian M.},
     TITLE = {Lie superalgebras and enveloping algebras},
    SERIES = {Graduate Studies in Mathematics},
    VOLUME = {131},
 PUBLISHER = {American Mathematical Society, Providence, RI},
      YEAR = {2012},
     PAGES = {xx+488},
      ISBN = {978-0-8218-6867-6},
   MRCLASS = {17-02 (16S30 17B35)},
  MRNUMBER = {2906817},
MRREVIEWER = {Aleksandr\ Nikolaevich\ Sergeev},
       DOI = {10.1090/gsm/131},
       URL = {https://doi.org/10.1090/gsm/131},
}

@article{McGehee1974,
  author       = {McGehee, Richard},
  title        = {Triple collision in the collinear three‑body problem},
  journal      = {Inventiones Mathematicae},
  volume       = {27},
  pages        = {191--227},
  year         = {1974},
  doi          = {10.1007/BF01390175}
}

@misc{scott2013geometrybkmanifolds,
      title={The Geometry of {$b^k$} Manifolds}, 
      author={Geoffrey Scott},
      year={2013},
      eprint={1304.3821},
      archivePrefix={arXiv},
      primaryClass={math.SG},
      url={https://arxiv.org/abs/1304.3821}, 
}

@article{Guillemin2017,
   title={{Desingularizing {$\boldsymbol{b^m}$}-Symplectic Structures}},
   volume={2019},
   ISSN={1687-0247},
   url={http://dx.doi.org/10.1093/imrn/rnx126},
   DOI={10.1093/imrn/rnx126},
   number={10},
   journal={International Mathematics Research Notices},
   publisher={Oxford University Press (OUP)},
   author={Guillemin, Victor and Miranda, Eva and Weitsman, Jonathan},
   year={2017},
   month=jun, pages={2981–2998} }

@book {MR1300632,
    AUTHOR = {Chari, Vyjayanthi and Pressley, Andrew},
     TITLE = {A guide to quantum groups},
 PUBLISHER = {Cambridge University Press, Cambridge},
      YEAR = {1994},
     PAGES = {xvi+651},
      ISBN = {0--521--43305--3},
   MRCLASS = {17B37 (16W30 81R50)},
  MRNUMBER = {1300632},
MRREVIEWER = {Tomasz\ Brzezi\'nski},
}

@misc{braddell2017invitationsingularsymplecticgeometry,
      title={An Invitation to Singular Symplectic Geometry}, 
      author={Roisin Braddell and Amadeu Delshams and Eva Miranda and Cédric Oms and Arnau Planas},
      year={2017},
      eprint={1705.03846},
      archivePrefix={arXiv},
      primaryClass={math.SG},
      url={https://arxiv.org/abs/1705.03846}, 
}

\end{document}